\documentclass[a4paper,12pt]{article}
\usepackage[utf8]{inputenc}
\usepackage{latexsym}
\usepackage{amsmath,amsthm,amsfonts}
\usepackage{graphicx}
\usepackage[hyperfootnotes=false,colorlinks=true,pagebackref=true]{hyperref}
\usepackage{multicol}
\usepackage{longtable}
\usepackage{subfigure}
\usepackage[margin=3cm]{geometry}
\usepackage{tikz}
\usetikzlibrary{calc}
\usetikzlibrary{arrows}

\newcommand{\C}{\mathcal{C}}
\newcommand{\I}{\mathcal{I}}

\renewcommand{\P}{\mathcal{P}}

\newcommand{\B}{\mathcal{B}}
\newcommand{\ZZ}{\mathbb{Z}}

\newtheorem{theorem}{Theorem}
\newtheorem{proposition}[theorem]{Proposition}
\newtheorem{corollary}[theorem]{Corollary}
\newtheorem{lemma}[theorem]{Lemma}
\newtheorem{conjecture}{Conjecture}

\theoremstyle{definition}
\newtheorem{definition}{Definition}

\newtheorem{example}{Example}

\begin{document}

\begin{center}
{\bf\Large Splittable and unsplittable graphs\\ and configurations}

\vspace{1.5em}
{\large Nino Ba\v{s}i\'{c}} \\[0.2em]
{\small\it FAMNIT, University of Primorska, \\
Glagoljaška 8, 6000 Koper, Slovenia} \\[0.1em]
\texttt{nino.basic@famnit.upr.si}

\vspace{1.5em}
{\large Jan Gro\v{s}elj} \\[0.2em]
{\small\it Faculty of Mathematics and Physics, University of Ljubljana, \\
 Jadranska 19, 1000 Ljubljana, Slovenia} \\[0.1em]
\texttt{jan.groselj@fmf.uni-lj.si}

\vspace{1.5em}
{\large Branko Gr\"{u}nbaum} \\[0.2em]
{\small\it Department of Mathematics, University of Washington, \\
Box 354350, Seattle, WA 98195, USA} \\[0.1em]
\texttt{grunbaum@math.washington.edu} 

\vspace{1.5em}
{\large Toma\v{z} Pisanski} \\[0.2em]
{\small\it FAMNIT, University of Primorska, \\
Glagoljaška 8, 6000 Koper, Slovenia} \\[0.1em]
\texttt{tomaz.pisanski@upr.si}
\end{center}

\begin{abstract}
We prove that there exist infinitely many splittable and also infinitely many unsplittable cyclic $(n_3)$ configurations.
We also present a complete study of trivalent cyclic Haar graphs on at most 60 vertices with respect to splittability.
Finally, we show that all cyclic flag-transitive configurations with the exception of the Fano plane and the Möbius-Kantor
configuration are splittable.
\end{abstract}

{\small 
\noindent
\textbf{Keywords:}
configuration of points and lines,
unsplittable configuration,
unsplittable graph,
independent set,
Levi graph,
Gr\" unbaum graph,
splitting type,
cyclic Haar graph.
\medskip

\noindent
\textbf{Math.\ Subj.\ Class.\ (2010):} 
51A20,
05B30
}

\bigskip

\section{Introduction and preliminaries}\label{sec:introduction}

The idea of \emph{unsplittable configuration} was conceived in 2004 and formally introduced in the monograph \cite{G} by Gr\"{u}nbaum. Later, it was also used in \cite{PS}. In \cite{PT}, the notion was generalized to graphs. In this paper we present some constructions for splittable and unsplittable cyclic configurations. In \cite{HMP}, the notion of \emph{cyclic Haar graph} was introduced. It was shown that cyclic Haar graphs are closely related to cyclic configurations. Namely, each cyclic Haar graph of girth $6$ is a Levi graph of a cyclic combinatorial configuration; see also \cite{PR}. For the definition of the \emph{Levi graph} (also called \emph{incidence graph}) of a configuration the reader is referred to \cite{C}. The classification of configurations with respect to splittability is a purely combinatorial problem and can be interpreted purely in terms of Levi graphs.

Let $n$ be a positive integer, let $\ZZ_n$ be the cyclic group of integers modulo $n$ and let $S \subseteq \ZZ_n$ be a set, called the \emph{symbol}.  The graph $H(n,S)$ with the vertex set $\{u_i \mid i \in \ZZ_n\} \cup \{v_i \mid i \in \ZZ_n\}$ and edges joining $u_i$ to $v_{i+k}$ for each $i \in \ZZ_n$ and each $k \in S$ is called a \emph{cyclic Haar graph} over $\ZZ_n$ with symbol $S$ \cite{HMP}. In practice, we will simplify the notation by denoting $u_i$ by $i^+$ and $v_i$ by $i^-$.

\begin{definition}
\label{def:combconf}
A \emph{combinatorial $(v_k)$ configuration} is an incidence structure  $\C = (\P,\B,\I)$, where $\I \subseteq \P \times \B$, $\P \cap \B = \emptyset$ and $|\P| = |\B| = v$. The elements of $\P$ are called \emph{points}, the elements of $\B$ are called \emph{lines} and the relation $\I$ is called the \emph{incidence} relation. Furthermore, each line is incident with $k$ points, each point is incident with $k$ lines and two distinct points are incident with at most one common line, i.e.,
\begin{equation}
\{ (p_1, b_1), (p_2, b_1), (p_1, b_2), (p_2, b_2) \} \subseteq \I, p_1 \neq p_2 \implies b_1 = b_2.
\label{eq:girthCondition}
\end{equation}
\end{definition}

If $(p, b) \in \I$ then we say that the line $b$ passes through point $p$ or that the point $p$ lies on line $b$.
An element of $\P \cup \B$ is called an \emph{element of configuration} $\C$.

A combinatorial $(v_k)$ configuration
$\C = (\P, \B, \I)$ is \emph{geometrically realisable} if the elements of $\P$ can be mapped to different
points in the Euclidean plane and the elements of $\B$ can be mapped to different lines in the Euclidean plane,
such that $(p, b) \in \I$ if and only if the point that corresponds to $p$ lies on the line that corresponds to $b$. 
A geometric realisation of a combinatorial $(v_k)$ configuration is called a \emph{geometric $(v_k)$ configuration}.
Note that examples in Figures~\ref{fig:pointSplittable}, \ref{fig:lineSplittable} and \ref{fig:splittableType4} are
all geometric configurations. The Fano plane $(7_3)$ is an example of a geometrically non-realizable configuration.

An \emph{isomorphism} between configurations $(\P, \B, \I)$ and $(\P', \B', \I')$ is a pair of bijections $\psi \colon \P \to \P'$ and
$\varphi \colon \B \to \B'$, such that
\begin{equation}
(p, b) \in \I \text{ if and only if } (\psi(p), \varphi(b)) \in \I'.
\end{equation}
The configuration $\C^* = (\B, \P, \I^*)$, where $\I^* = \{(b, p) \in \B \times \P \mid (p, b) \in \I \}$, is called the
\emph{dual configuration} of $\C$.
A configuration that is isomorphic to its dual is called a \emph{self-dual} configuration.

The \emph{Levi graph} of a configuration $\C$ is the bipartite
graph on the vertex set $\P \cup \B$ having an edge between $p \in \P$ and $b \in \B$ if and only if the elements
$p$ and $b$ are incident in $\C$, i.e., if $(p,b) \in \I$. It is denoted $L(\C)$.
Condition \eqref{eq:girthCondition} in Definition~\ref{def:combconf} implies that the girth of $L(\C)$ is at least 6.
Moreover, any combinatorial $(v_k)$ configuration is completely determined by a $k$-regular bipartite graph of girth at least 6 with a given
black-and-white vertex coloring, where black vertices correspond to points and white vertices correspond to lines. Such a graph will be called a \emph{colored Levi graph}. Note that the reverse coloring determines the dual configuration $\C^* = (\B, \P, \I^*)$. Also, an isomorphism between configurations corresponds to
color-preserving isomorphism between their respective colored Levi graphs.

A configuration $\C$ is said to be \emph{connected} if its Levi graph $L(\C)$ is connected. Similarly, a configuration $\C$
is said to be \emph{$k$-connected} if its Levi graph $L(\C)$ is $k$-connected.

\begin{definition}[\cite{PS}]
A combinatorial $(v_k)$ configuration $\C$ is \emph{cyclic} if admits an automorphism of order $v$ that cyclically permutes
the points and lines, respectively.
\end{definition}

In \cite{HMP} the following was proved:
\begin{proposition}
A configuration $\C$ is cyclic if and only if its Levi graph is isomorphic to a cyclic Haar graph of girth 6.
\end{proposition}
It can be shown that each cyclic configuration is self-dual, see for instance~\cite{HMP}. 

\section{Splittable and unsplittable configurations (and graphs)}
\label{sec:definitions}
Let $G$ be any graph. The \emph{square} of $G$, denoted $G^2$, is a graph with the same vertex set as $G$, where two vertices are adjacent  if and only if their distance
in $G$ is at most 2. In other words, $V(G^2) = V(G)$ and $E(G^2) = \{uv \mid d_G(u, v) \leq 2 \}$. The square of the Levi graph $L(\C)$ of
a configuration $\C$ is called the \emph{Gr\"{u}nbaum graph} of $\C$ in \cite{PS} and~\cite{PT}. In \cite{G}, it is called the \emph{independence graph}. Two elements of a configuration $\C$ are said to be \emph{independent} if they correspond
to independent vertices of the Gr\"{u}nbaum graph.

\begin{example}
The Grünbaum graph of the Heawood graph is shown in Figure~\ref{fig:3connSplittable}. Its complement is the Möbius ladder $M_{14}$.
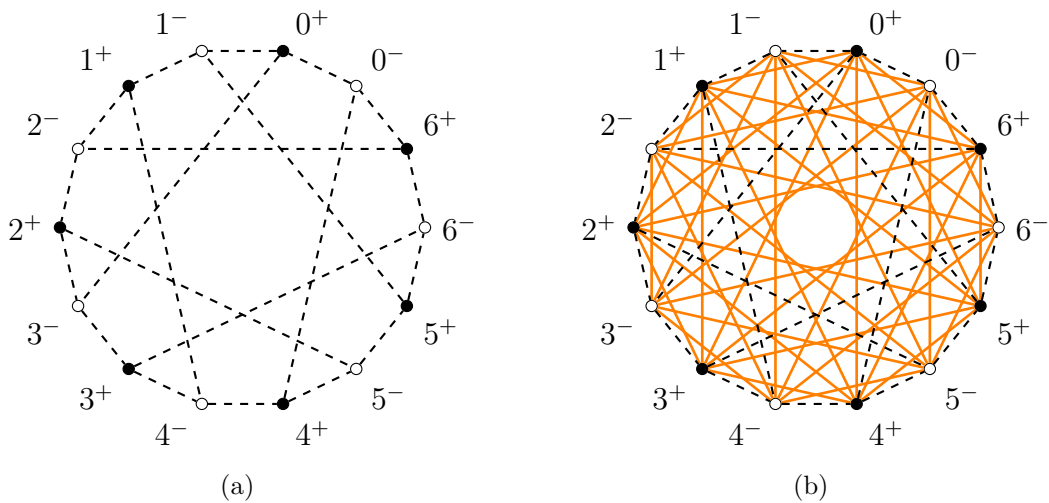
\begin{figure}[!htb]
\centering
\subfigure[]{
\usetikzlibrary{calc}
\begin{tikzpicture}[scale=1.2]
\pgfmathsetmacro{\a}{360 / 14}
\tikzstyle{white} = [inner sep=1.5, draw, circle, fill=white]
\tikzstyle{black} = [inner sep=1.5, draw, circle, fill=black]
\tikzstyle{edge} = [draw, dashed, line width=0.8, color=black]
\node[white,label=0:$6^-$] (a0) at (0:2) {};
\node[black,label=\a:$6^+$] (a1) at (\a:2) {};
\node[white,label=2*\a:$0^-$] (a2) at (2*\a:2) {};
\node[black,label=3*\a:$0^+$] (a3) at (3*\a:2) {};
\node[white,label=4*\a:$1^-$] (a4) at (4*\a:2) {};
\node[black,label=5*\a:$1^+$] (a5) at (5*\a:2) {};
\node[white,label=6*\a:$2^-$] (a6) at (6*\a:2) {};
\node[black,label=7*\a:$2^+$] (a7) at (7*\a:2) {};
\node[white,label=8*\a:$3^-$] (a8) at (8*\a:2) {};
\node[black,label=9*\a:$3^+$] (a9) at (9*\a:2) {};
\node[white,label=10*\a:$4^-$] (a10) at (10*\a:2) {};
\node[black,label=11*\a:$4^+$] (a11) at (11*\a:2) {};
\node[white,label=12*\a:$5^-$] (a12) at (12*\a:2) {};
\node[black,label=13*\a:$5^+$] (a13) at (13*\a:2) {};
\draw[edge] (a0) -- (a1) -- (a2) -- (a3) -- (a4) -- (a5) -- (a6)  -- (a7)  -- (a8)  -- (a9) -- (a10)  -- (a11)  -- (a12)  -- (a13) -- (a0);
\draw[edge] (a3) -- (a8);
\draw[edge] (a5) -- (a10);
\draw[edge] (a7) -- (a12);
\draw[edge] (a9) -- (a0);
\draw[edge] (a11) -- (a2);
\draw[edge] (a13) -- (a4);
\draw[edge] (a1) -- (a6);
\end{tikzpicture}
}
\quad
\subfigure[]{
\usetikzlibrary{calc}
\begin{tikzpicture}[scale=1.2]
\pgfmathsetmacro{\a}{360 / 14}
\tikzstyle{white} = [inner sep=1.5, draw, circle, fill=white]
\tikzstyle{black} = [inner sep=1.5, draw, circle, fill=black]
\tikzstyle{edge} = [draw, dashed, line width=0.8, color=black]
\tikzstyle{oedge} = [draw, line width=1, color=orange]
\node[white,label=0:$6^-$] (a0) at (0:2) {};
\node[black,label=\a:$6^+$] (a1) at (\a:2) {};
\node[white,label=2*\a:$0^-$] (a2) at (2*\a:2) {};
\node[black,label=3*\a:$0^+$] (a3) at (3*\a:2) {};
\node[white,label=4*\a:$1^-$] (a4) at (4*\a:2) {};
\node[black,label=5*\a:$1^+$] (a5) at (5*\a:2) {};
\node[white,label=6*\a:$2^-$] (a6) at (6*\a:2) {};
\node[black,label=7*\a:$2^+$] (a7) at (7*\a:2) {};
\node[white,label=8*\a:$3^-$] (a8) at (8*\a:2) {};
\node[black,label=9*\a:$3^+$] (a9) at (9*\a:2) {};
\node[white,label=10*\a:$4^-$] (a10) at (10*\a:2) {};
\node[black,label=11*\a:$4^+$] (a11) at (11*\a:2) {};
\node[white,label=12*\a:$5^-$] (a12) at (12*\a:2) {};
\node[black,label=13*\a:$5^+$] (a13) at (13*\a:2) {};
\draw[oedge] (a1) -- (a3);
\draw[oedge] (a2) -- (a4);
\draw[oedge] (a3) -- (a5);
\draw[oedge] (a4) -- (a6);
\draw[oedge] (a5) -- (a7);
\draw[oedge] (a6) -- (a8);
\draw[oedge] (a7) -- (a9);
\draw[oedge] (a8) -- (a10);
\draw[oedge] (a9) -- (a11);
\draw[oedge] (a10) -- (a12);
\draw[oedge] (a11) -- (a13);
\draw[oedge] (a12) -- (a0);
\draw[oedge] (a13) -- (a1);
\draw[oedge] (a0) -- (a2);
\draw[oedge] (a0) -- (a4);
\draw[oedge] (a1) -- (a5);
\draw[oedge] (a2) -- (a6);
\draw[oedge] (a3) -- (a7);
\draw[oedge] (a4) -- (a8);
\draw[oedge] (a5) -- (a9);
\draw[oedge] (a6) -- (a10);
\draw[oedge] (a7) -- (a11);
\draw[oedge] (a8) -- (a12);
\draw[oedge] (a9) -- (a13);
\draw[oedge] (a10) -- (a0);
\draw[oedge] (a11) -- (a1);
\draw[oedge] (a12) -- (a2);
\draw[oedge] (a13) -- (a3);
\draw[oedge] (a0) -- (a6);
\draw[oedge] (a1) -- (a7);
\draw[oedge] (a2) -- (a8);
\draw[oedge] (a3) -- (a9);
\draw[oedge] (a4) -- (a10);
\draw[oedge] (a5) -- (a11);
\draw[oedge] (a6) -- (a12);
\draw[oedge] (a7) -- (a13);
\draw[oedge] (a8) -- (a0);
\draw[oedge] (a9) -- (a1);
\draw[oedge] (a10) -- (a2);
\draw[oedge] (a11) -- (a3);
\draw[oedge] (a12) -- (a4);
\draw[oedge] (a13) -- (a5);
\draw[edge] (a0) -- (a1) -- (a2) -- (a3) -- (a4) -- (a5) -- (a6)  -- (a7)  -- (a8)  -- (a9) -- (a10)  -- (a11)  -- (a12)  -- (a13) -- (a0);
\draw[edge] (a3) -- (a8);
\draw[edge] (a5) -- (a10);
\draw[edge] (a7) -- (a12);
\draw[edge] (a9) -- (a0);
\draw[edge] (a11) -- (a2);
\draw[edge] (a13) -- (a4);
\draw[edge] (a1) -- (a6);
\end{tikzpicture}
}
\caption{The Heawood graph $H = H(7,\{0,1,3\}) \cong \mathrm{LCF}[5,-5]^7$ (on the left) is the Levi graph of the Fano plane. Its Gr\"{u}nbaum graph $G$ is on the right. Note that there is an orange solid edge between two vertices of $G$ if and only if they are
at distance 2 in $H$. }
\label{fig:3connSplittable}
\end{figure}
\end{example}

It is easy to see that two elements of $\C$ are independent if and only if one of the following hold:
\begin{enumerate}
\renewcommand{\labelenumi}{(\roman{enumi})}
\item two points of $\C$ that do not lie on a common line of $\C$;
\item two lines of $\C$ that do not intersect in a common point of $\C$;
\item a point of $\C$ and a line of $\C$ that are not incident.
\end{enumerate}    

The definition of unsplittable configuration was introduced in \cite{G} and is equivalent to the following:

\begin{definition}
A configuration $\C$ is \emph{splittable} if there exists an independent set of vertices $\Sigma$ in the Gr\"{u}nbaum graph
$(L(\C))^2$ such that $L(\C) - \Sigma$, i.e., the graph obtained by removing the set of vertices $\Sigma$ from the Levi graph $L(\C)$, is disconnected.
In this case the set $\Sigma$ is called a \emph{splitting set of elements}.  A configuration that is not splittable is called \emph{unsplittable}.
\end{definition}
This definition carries over to graphs:
\begin{definition}
A connected graph $G$ is \emph{splittable} if there exists an independent set $\Sigma$ in $G^2$ such
that $G - \Sigma$ is disconnected.
\end{definition}

\begin{example}
Every cycle of length at least 6 is splittable (there exists a pair of vertices at distance 3 in $G$).

Every graph of diameter 2 without a cut vertex is unsplittable. The square of such a graph on $n$
vertices is the complete graph $K_n$. This implies that $|S| = 1$. Since there are no cut vertices,
a splitting set does not exist. The Petersen graph is an example of unsplittable graph.
\end{example}

In~\cite{G}, refinements of the above definition are also considered. Configuration $\C$ is 
\emph{point-splittable} if it is splittable and there exists a splitting set of elements that consists of points only (i.e., only black vertices in
the corresponding colored Levi graph). In a similar way
\emph{line-splittable} configurations are defined. Note that these refinements can be defined for any bipartite graph
with a given black-and-white coloring. There are four possibilities, that we call \emph{splitting types}.
Any configuration may be:
\begin{enumerate}
\renewcommand{\labelenumi}{(T\arabic{enumi})}
\item point-splittable, line-splittable,
\item point-splittable, line-unsplittable,
\item point-unsplittable, line-splittable,
\item point-unsplittable, line-unsplittable.
\end{enumerate}    

Any configuration of splitting type T1, T2 or T3 is splittable. A configuration of splitting type T4 may be splittable or 
unsplittable. For an example of a point-splittable (T2) configuration see Figure~\ref{fig:pointSplittable}. The configuration on
Figure~\ref{fig:pointSplittable} is isomorphic to a configuration on Figure 5.1.11 from \cite{G}. For an example
of a line-splittable (T3) configuration see Figure~\ref{fig:lineSplittable}.
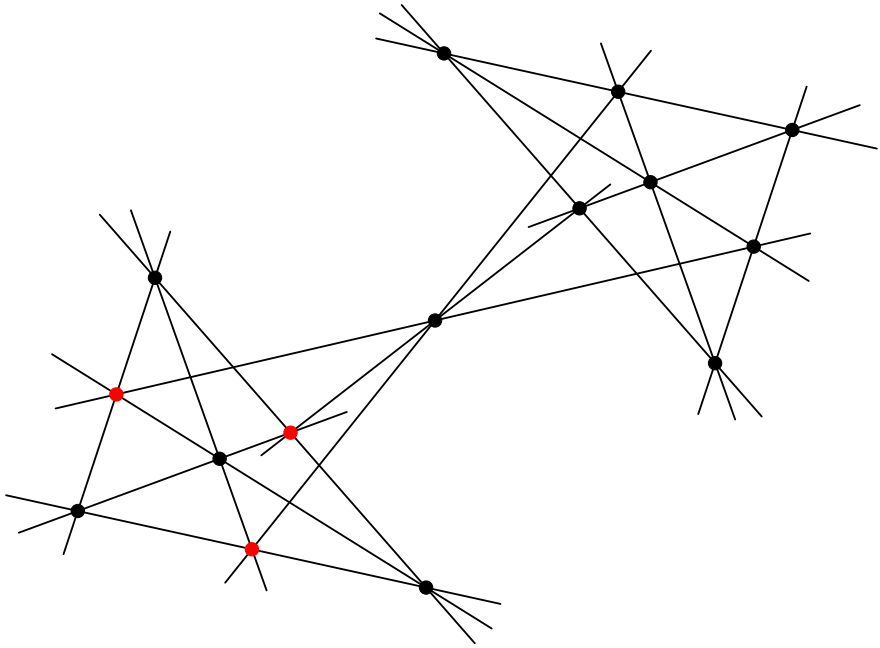
\begin{figure}[!htb]
\centering
\definecolor{uuuuuu}{rgb}{0.0,0.0,0.0}
\definecolor{ududff}{rgb}{1.0,0.0,0.0}
\begin{tikzpicture}[line cap=round,line join=round,>=triangle 45,x=1.0cm,y=1.0cm]
\clip(1.8592195121951212,2.0788292682926808) rectangle (13.569658536585376,10.767219512195112);
\draw [line width=.7pt] (2.7337026712107106,3.371390270602799)-- (4.137992602932553,7.6495758765460735);
\draw [line width=.7pt] (2.5820777313892247,6.0233096093199014)-- (8.365477685865843,2.38563485534475);
\draw [line width=.7pt] (3.6200006787025276,7.9317632142784165)-- (5.404879540698868,2.8911330947517238);
\draw [line width=.7pt] (1.9781879106418252,4.153119777915322)-- (8.480845708300448,2.711809028949742);
\draw [line width=.7pt] (3.2096476670588228,7.874878621095507)-- (8.143110713603862,2.189960938321759);
\draw [line width=.7pt] (6.844803509114219,10.209694671037425)-- (13.431206395935057,8.749821866226611);
\draw [line width=.7pt] (6.893055796398109,10.543195657790465)-- (12.534815044268154,6.994610855506983);
\draw [line width=.7pt] (7.180472376807803,10.654212576942848)-- (11.916671663962632,5.196605451347222);
\draw [line width=.7pt] (9.801250498985603,10.14500598694333)-- (11.567489522470998,5.157016152100319);
\draw [line width=.7pt] (2.1461499815030263,3.6560474563407372)-- (6.459634837749807,5.257678964145437);
\draw [line width=.7pt] (12.507097519323047,9.571535324040882)-- (11.081602538456526,5.228748289308005);
\draw [line width=.7pt] (13.205870792805461,9.326034560094891)-- (8.850706660379409,7.708927202907171);
\draw [line width=.7pt] (4.8598525235798045,2.992999287537109)-- (10.460267802680622,10.048424418561183);
\draw [line width=.7pt] (2.6296133461688216,5.303477445346096)-- (12.555146709572332,7.624095104958757);
\draw [line width=.7pt] (5.335223286535328,4.6821577788579685)-- (9.925784691047372,8.274771051954367);
\draw [fill=uuuuuu] (3.936878048780487,7.036878048780478) circle (2.5pt);
\draw [fill=uuuuuu] (2.921658536585365,3.944) circle (2.5pt);
\draw [fill=uuuuuu] (7.501951219512194,2.9287804878048713) circle (2.5pt);
\draw [fill=ududff,color=ududff] (3.429268292682926,5.4904390243902395) circle (2.5pt);
\draw [fill=ududff,color=ududff] (5.71941463414634,4.982829268292675) circle (2.5pt);
\draw [fill=ududff,color=ududff] (5.211804878048779,3.4363902439024354) circle (2.5pt);
\draw [fill=uuuuuu] (7.62,6.470243902439024) circle (2.5pt);
\draw [fill=uuuuuu] (4.786829268292681,4.63655284552845) circle (2.5pt);
\draw [fill=uuuuuu] (11.303121951219513,5.90360975609757) circle (2.5pt);
\draw [fill=uuuuuu] (12.318341463414635,8.996487804878047) circle (2.5pt);
\draw [fill=uuuuuu] (7.738048780487806,10.011707317073176) circle (2.5pt);
\draw [fill=uuuuuu] (10.453170731707319,8.303934959349597) circle (2.5pt);
\draw [fill=uuuuuu] (11.810731707317077,7.450048780487808) circle (2.5pt);
\draw [fill=uuuuuu] (10.028195121951217,9.504097560975614) circle (2.5pt);
\draw [fill=uuuuuu] (9.520585365853657,7.957658536585377) circle (2.5pt);
\end{tikzpicture}
\caption{A point-splittable $(15_3)$ configuration of type 2. Points that belong to a splitting set are colored red. Its dual is of type 3 (see Figure~\ref{fig:lineSplittable}).}
\label{fig:pointSplittable}
\end{figure}
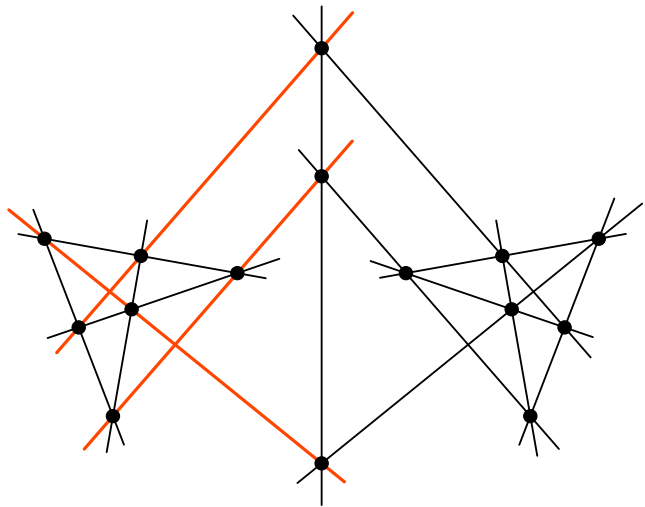
\begin{figure}[!htb]
\centering
\definecolor{uuuuuu}{rgb}{0.0,0.0,0.0}
\definecolor{pomaranca}{rgb}{1.0,0.271,0.0}
\begin{tikzpicture}[line cap=round,line join=round,>=triangle 45,x=1.0cm,y=1.0cm]
\clip(0.2382415373244706,-0.2978888396156705) rectangle (8.72201389504807,6.454501404286756);
\draw [line width=.7pt] (0.6275531658143828,3.6712372542538456)-- (1.8204297871306918,0.5545342637691801);
\draw [line width=1.2pt,color=pomaranca] (1.2981879693747151,0.4967954568586148)-- (4.822663176736382,4.580711173325307);
\draw [line width=.7pt] (0.4310878346985034,3.3479559869367126)-- (3.6856943870583656,2.762268180430754);
\draw [line width=.7pt] (1.5923578871778972,0.4549721712150598)-- (2.1247458292818395,3.528587263703896);
\draw [line width=.7pt] (0.8155212295587199,1.9686126660638288)-- (3.8642819096588803,3.0203258467321077);
\draw [line width=.7pt] (4.42,6.367932298595699)-- (4.42,-0.24594737620103535);
\draw [line width=1.2pt,color=pomaranca] (0.30616138548773747,3.669213404328374)-- (4.719665500804437,0.062207178643628236);
\draw [line width=1.2pt,color=pomaranca] (0.9361915575228804,1.7751118612956178)-- (4.826857433740685,6.283343749611484);
\draw [line width=.7pt] (5.186621989019631,2.7680837243728074)-- (8.422339176854669,3.350372265776564);
\draw [line width=.7pt] (7.018946477745434,0.5529045921911577)-- (8.269278867969973,3.8197258460208423);
\draw [line width=.7pt] (5.06291474970014,2.9902461239920757)-- (7.98735491386069,1.981419066230917);
\draw [line width=.7pt] (6.715226751539854,3.5287455618049632)-- (7.255959550356794,0.40695342489178365);
\draw [line width=.7pt] (4.118874154277754,4.463056805642259)-- (7.540994289213668,0.49774299849429715);
\draw [line width=.7pt] (8.636078172395868,3.7527703042795055)-- (4.100906070994731,0.04632898760887505);
\draw [line width=.7pt] (4.045822848550674,6.245476120924671)-- (7.958660769110559,1.711552816148933);
\draw [fill=uuuuuu] (0.7749699926090192,3.286072135994082) circle (2.5pt);
\draw [fill=uuuuuu] (1.6752886917960068,0.9337534368070953) circle (2.5pt);
\draw [fill=uuuuuu] (3.3114447893569827,2.829616851441241) circle (2.5pt);
\draw [fill=uuuuuu] (4.42,4.114133206630769) circle (2.5pt);
\draw [fill=uuuuuu] (2.043207390983001,3.0578444937176616) circle (2.5pt);
\draw [fill=uuuuuu] (1.225129342202513,2.1099127864005887) circle (2.5pt);
\draw [fill=uuuuuu] (1.9205678245873359,2.34981414141414) circle (2.5pt);
\draw [fill=uuuuuu] (4.42,5.811905770832596) circle (2.5pt);
\draw [fill=uuuuuu] (4.42,0.3071135611421479) circle (2.5pt);
\draw [fill=uuuuuu] (5.528555210643017,2.829616851441241) circle (2.5pt);
\draw [fill=uuuuuu] (8.06503000739098,3.2860721359940825) circle (2.5pt);
\draw [fill=uuuuuu] (7.164711308203993,0.9337534368070957) circle (2.5pt);
\draw [fill=uuuuuu] (7.6148706577974865,2.109912786400589) circle (2.5pt);
\draw [fill=uuuuuu] (6.9194321754126635,2.3498141414141402) circle (2.5pt);
\draw [fill=uuuuuu] (6.796792609017,3.057844493717662) circle (2.5pt);
\end{tikzpicture}
\caption{A line-splittable $(15_3)$ configuration of type 3.
Lines that belong to a splitting set are colored orange.
Its dual is depicted in Figure~\ref{fig:pointSplittable}.}
\label{fig:lineSplittable}
\end{figure}

Note the following:
\begin{proposition}
If $\C$ is of type 1 then its dual is also of type 1. If it is of type 2 then its dual is of type 3 (and vice versa). If it is of type 4
then its dual is also of type 4.
\end{proposition}

Since types are mutually disjoint, this has a straightforward consequence for cyclic configurations:

\begin{corollary}
Any self-dual configuration, in particular any cyclic configuration, is either of type 1 or 4.
\end{corollary}

Obviously, unsplittable configurations are of type 4. However, the converse is not true:
\begin{proposition}
\label{prop:unsplitGp}
Any unsplittable configuration is point-unsplittable and line-unsplittable.
There exist splittable configurations that are both point-unsplittable and line-unsplittable.
\end{proposition} 
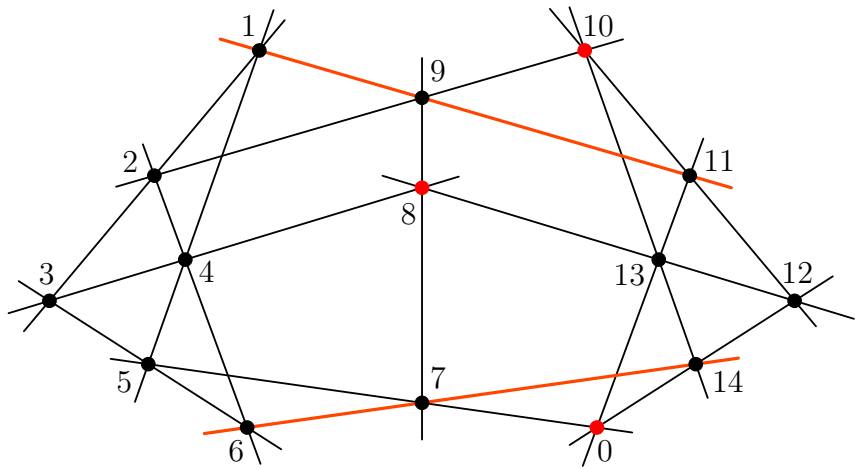
\begin{figure}[!htb]
\centering
\definecolor{uuuuuu}{rgb}{0.0,0.0,0.0}
\definecolor{pomaranca}{rgb}{1.0,0.271,0.0}
\definecolor{ududff}{rgb}{1.0,0.0,0.0}
\begin{tikzpicture}[line cap=round,line join=round,>=triangle 45,x=1.0cm,y=1.0cm]
\tikzstyle{every node}=[circle,draw,inner sep=1.8pt]
\clip(0.72,2.08) rectangle (12.04,8.4);
\draw [line width=.7pt] (0.9551744865586915,4.641579562531307)-- (4.407153114042411,2.411070295541827);
\draw [line width=.7pt] (4.1358791521670355,2.218494780132869)-- (2.5881657703258467,6.45567731730465);
\draw [line width=.7pt] (1.0244171673819755,3.976327896995709)-- (4.453806008583691,8.101534763948498);
\draw [line width=.7pt] (2.484354125079744,3.039529561871052)-- (4.3076039758813,8.234542835387813);
\draw [line width=.7pt] (0.8214615267097823,4.215223302948516)-- (6.743883485257424,6.027307633548909);
\draw [line width=.7pt] (8.2007580564368,2.4678744364668552)-- (11.662721656370012,4.704835531808314);
\draw [line width=.7pt] (10.017227995801864,3.09200790237277)-- (8.199025785606981,8.272638857448603);
\draw [line width=.7pt] (11.441567381974249,4.041303004291844)-- (7.941697854077255,8.251290987124463);
\draw [line width=.7pt] (8.374201834862387,2.1913394495412843)-- (9.959468522619426,6.531331857007276);
\draw [line width=.7pt] (2.2363928809584044,5.892298972153878)-- (8.903607119041597,7.847701027846123);
\draw [line width=1.2pt,color=pomaranca] (3.605590609119547,7.850869185311229)-- (10.32805261100508,5.8792637218607);
\draw [line width=.7pt] (5.7395406630340675,6.038498752355244)-- (11.940375821153944,4.141228293527524);
\draw [line width=.7pt] (2.164351890436879,3.6105668494971224)-- (9.02481546137472,2.6338228834652937);
\draw [line width=1.2pt,color=pomaranca] (3.397171271216029,2.61986845217313)-- (10.420037842525538,3.6197342013087197);
\draw [line width=.7pt] (6.26,7.6)-- (6.26,2.54);
\node [fill=uuuuuu,label={[xshift=-1,yshift=0]$3$}] at (1.36,4.38) {};
\node [fill=uuuuuu,label={[xshift=-4,yshift=-19]$6$}] at (3.96,2.7) {};
\node [fill=uuuuuu,label={[xshift=-4,yshift=-1]$1$}] at (4.12,7.7) {};
\node [fill=uuuuuu,label={[xshift=-9,yshift=-5]$2$}] at (2.74,6.04) {};
\node [fill=uuuuuu,label={[xshift=-9,yshift=-17]$5$}] at (2.66,3.54) {};
\node [fill=uuuuuu,label={[xshift=8,yshift=-15]$4$}] at (3.1466666666666665,4.926666666666665) {};
\node [fill=ududff,color=ududff,label={[xshift=-5,yshift=-20]$8$}] at (6.26,5.8792537313432796) {};
\node [fill=ududff,color=ududff,label={[xshift=5,yshift=-3]$10$}] at (8.4,7.7) {};
\node [fill=ududff,color=ududff,label={[xshift=3,yshift=-19]$0$}] at (8.56,2.7) {};
\node [fill=uuuuuu,label={[xshift=1,yshift=-2]$12$}] at (11.16,4.38) {};
\node [fill=uuuuuu,label={[xshift=12,yshift=-19]$14$}] at (9.86,3.54) {};
\node [fill=uuuuuu,label={[xshift=11,yshift=-7]$11$}] at (9.78,6.04) {};
\node [fill=uuuuuu,label={[xshift=-11,yshift=-18]$13$}] at (9.373333333333335,4.926666666666664) {};
\node [fill=uuuuuu,label={[xshift=6]$9$}] at (6.26,7.072367491166078) {};
\node [fill=uuuuuu,label={[xshift=6,yshift=-1]$7$}] at (6.26,3.027457627118644) {};
\end{tikzpicture}
\caption{A splittable $(15_3)$ configuration of type 4. Elements of a splitting set are points 0, 8 and 10 (colored blue) and
lines 1 9 11 and 6 7 14 (colored green).}
\label{fig:splittableType4}
\end{figure}

\begin{proof}
The first statement of Proposition~\ref{prop:unsplitGp} is obviously true. An example that provides the proof
of the second statement is shown in Figure~\ref{fig:splittableType4}. The splitting set is \{0, 8, 10, (1, 9, 11), (6, 7, 14)\}.
\end{proof}

Note that configuration in Figure~\ref{fig:splittableType4} is not cyclic, but it is 3-connected. In \cite{G}, the following theorem is proven:
\begin{theorem}[{\cite[Theorem 5.1.5]{G}}]
\label{thm:unsplit3Conn}
Any unsplittable $(n_3)$ configuration is 3-connected.
\end{theorem}

Our computational results show that the converse to Theorem~\ref{thm:unsplit3Conn} is not true. There exist 3-connected splittable configurations.  See,
for instance, the configuration in Figure~\ref{fig:splittableType4}.
%

\section{Splittable and unsplittable cyclic $\boldsymbol{(n_3)}$ configurations}

We used a computer program to analyse all cyclic $(n_3)$ configurations for $7 \leq n \leq 30$
(see Table~\ref{tbl:numericsplit}, Table~\ref{default} and Table~\ref{default22}). In \cite{HMP} it was shown that cyclic Haar graphs contain all information about cyclic combinatorial configurations. In trivalent case 
combinatorial isomorphisms of cyclic configurations are well-understood; see \cite{KKP}. 
Namely, it is known how to obtain all sets of parameters of isomorphic cyclic Haar graphs.
We would like to draw the reader's attention to the manuscript \cite{kkmm}, where the main result of \cite{KKP}
is extended to cyclic $(n_k)$ configurations for all $k > 3$.
One would expect that large sparse graphs are splittable.
In this sense the following result is not a surprise:

\begin{theorem}
\label{lem:tech}
Let $H(n,\{0,a,b\})$ be a cyclic Haar graph, where $0 < a < b$. Let
\begin{align*}
\mathcal{W} & = \{ 0, a, b, 2b, b + a, b - a, 2b - a, 2b - 2a, 3b - a, 3b - 2a, 2b + a, 3b \} \\
\mathcal{B} & = \{ 0, a, b, 2b, b + a, b - a, 2b - a, 2b - 2a, 3b - a, 3b - 2a, -a, b - 2a \}
\end{align*}
be multisets with elements from $\mathbb{Z}_n$. If all elements of $\mathcal{W}$ are distinct and all elements of $\mathcal{B}$ are distinct, i.e.\ $\mathcal{W}$ and $\mathcal{B}$ are ordinary sets, $|\mathcal{W}| = |\mathcal{B}| = 12$,
then $H(n, \{0, a, b\})$ is splittable and
$$
\Sigma = \{ 0^+, 2b^+, (2b - 2a)^+, (b - a)^-, (b + a)^-, (3b - a)^- \}
$$
is a splitting set for $H(n, \{0, a, b\})$.
\end{theorem}

\begin{proof}
See Figure~\ref{fig:hexGeneral}. If $\mathcal{W}$ and $\mathcal{B}$ are ordinary sets then the graph in Figure~\ref{fig:hexGeneral} is a subgraph of $H(n,\{0,a,b\})$.
It is easy to see that $\Sigma$ is a splitting set. The set $\Sigma$ is indeed an independent set in the square of the graph $H(n,\{0,a,b\})$ since no two vertices of $\Sigma$ are adjacent to the same vertex.
In order to see that the subgraph obtained by removing the vertices of $\Sigma$ is disconnected, note that one
of the connected components is the cycle determined by vertices $\{ b^+, (b - a)^+, (2b - a)^+, b^-, 2b^-, (2b - a)^- \}$.
\begin{figure}[!htbp]
\centering
\usetikzlibrary{calc}
\begin{tikzpicture}
\tikzstyle{white} = [inner sep=1.5, draw, circle, fill=white]
\tikzstyle{black} = [inner sep=1.5, draw, circle, fill=black]
\tikzstyle{edge} = [draw, line width=0.8, color=black]
\node[white,label=-30:$(2b-a)^-$] (7m) at (0, 0) {};
\node[black,label=30:$(b-a)^+$] (3p) at (30:1.5) {};
\node[white,label=-30:$b^-$] (4m) at ($ (3p) + (-30:1.5) $) {};
\node[black,label=30:$(2b-a)^+$] (7p) at (-90:1.5) {};
\node[white,label=-30:$2b^-$] (8m) at ($ (-90:1.5) + (-30:1.5) $) {};
\node[black,label=30:$b^+$] (4p) at ($ (4m) + (-90:1.5) $) {};
\draw[edge] (8m) -- (7p) -- (7m) -- (3p) -- (4m) -- (4p) -- (8m);
\node[white,label=-30:$(b-a)^-$] (3m) at ($ (3p) + (90:1.5) $) {};
\node[black,label=30:$0^+$] (0p) at ($ (4m) + (30:1.5) $) {};
\node[white,label=-30:$(b+a)^-$] (5m) at ($ (4p) + (-30:1.5) $) {};
\node[black,label=30:$2b^+$] (8p) at ($ (8m) + (-90:1.5) $) {};
\node[white,label=-30:$(3b-a)^-$] (11m) at ($ (7p) + (210:1.5) $) {};
\node[black,label=30:$(2b-2a)^+$] (6p) at ($ (7m) + (150:1.5) $) {};
\draw[edge] (11m) -- (7p);
\draw[edge] (3m) -- (3p);
\draw[edge] (4m) -- (0p);
\draw[edge] (5m) -- (4p);
\draw[edge] (8m) -- (8p);
\draw[edge] (7m) -- (6p);
\node[black,label=-180:$(b-2a)^+$] (2p) at ($ (3m) + (150:1.5) $) {};
\node[black,label=0:$-a^+$] (nm) at ($ (3m) + (30:1.5) $) {};
\draw[edge] (2p) -- (3m) -- (nm);
\node[white,label=0:$a^-$] (1m) at ($ (0p) + (90:1.5) $) {};
\node[white,label=0:$0^-$] (0m) at ($ (0p) + (-30:1.5) $) {};
\draw[edge] (1m) -- (0p) -- (0m);
\node[black,label=0:$a^+$] (1p) at ($ (5m) + (30:1.5) $) {};
\node[black,label=0:$(b+a)^+$] (5p) at ($ (5m) + (-90:1.5) $) {};
\draw[edge] (1p) -- (5m) -- (5p);
\node[white,label=0:$3b^-$] (9m) at ($ (8p) + (-30:1.5) $) {};
\node[white,label=180:$(2b+a)^-$] (12m) at ($ (8p) + (210:1.5) $) {};
\draw[edge] (12m) -- (8p) -- (9m);
\node[black,label=180:$(3b-a)^+$] (11p) at ($ (11m) + (-90:1.5) $) {};
\node[black,label=180:$(3b-2a)^+$] (10p) at ($ (11m) + (150:1.5) $) {};
\draw[edge] (11p) -- (11m) -- (10p);
\node[white,label=180:$(3b-2a)^-$] (10m) at ($ (6p) + (90:1.5) $) {};
\node[white,label=180:$(2b-2a)^-$] (6m) at ($ (6p) + (210:1.5) $) {};
\draw[edge] (10m) -- (6p) -- (6m);
\end{tikzpicture}
\caption{The set $\Sigma = \{ 0^+, 2b^+, (2b - 2a)^+, (b - a)^-, (b + a)^-, (3b - a)^-\}$ is a splitting set for $H(n,\{0,a,b\})$.}
\label{fig:hexGeneral}
\end{figure}
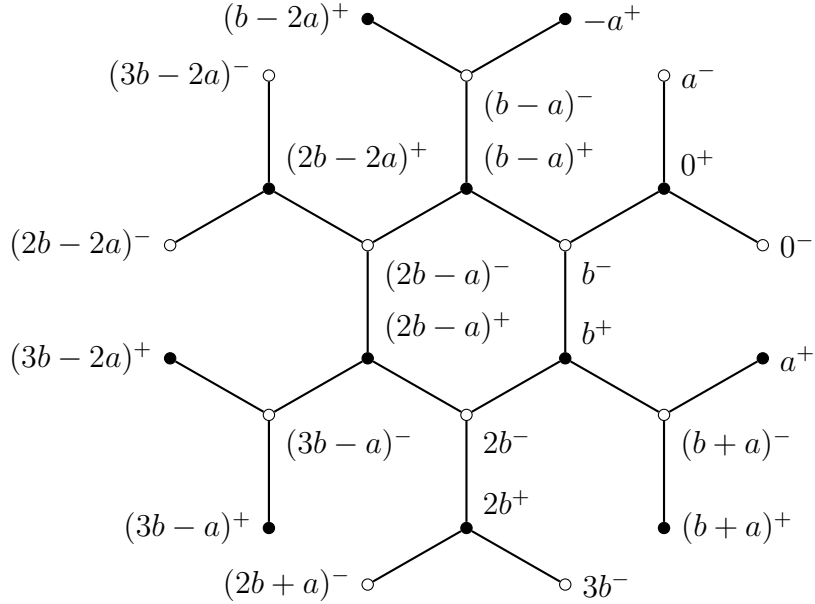
\end{proof}

\begin{corollary}
\label{cor:g6}
Under conditions of Theorem~\ref{lem:tech}, the girth of the graph $H(n,\{0,a,b\})$ is~6.
\end{corollary}

\begin{proof}
The girth of such a graph is at most 6 because it contains a $6$-cycle (see Figure~\ref{fig:hexGeneral}). It is easy to
see that the girth cannot be $4$. Because the graph $H(n,\{0,a,b\})$ is bipartite, each 4-cycle must contain a black vertex.
Consider vertex $b^+$ in Figure~\ref{fig:hexGeneral}. Its neighborhood is $\{ b^-, 2b^-, (b+a)^- \}$. None of those
vertices have a common neighbor, so $b^+$ does not belong to any 4-cycle. Because of symmetry this argument holds for all
black vertices.
\end{proof}

\begin{corollary}
\label{thm:5}
There exist infinitely many cyclic $(n_3)$ configurations that are splittable. For example,
the following three families of cyclic Haar graphs are splittable:
\begin{enumerate}
\renewcommand{\labelenumi}{(\alph{enumi})}
\item $H(n,\{0,1,4\})$ for $n \ge 13$,
\item $H(n,\{0,1,5\})$ for $n \ge 16$, and
\item $H(n,\{0,2,5\})$ for $n \ge 16$.
\end{enumerate}
\end{corollary}

\begin{proof}
Corollary~\ref{cor:g6} implies that each graph from any of the three families has girth 6.
From Theorem~\ref{lem:tech} it follows that $\Sigma = \{0^+,6^+,8^+,3^-,5^-,11^-\}$ is a splitting set for $H(n,\{0,1,4\})$ if $n \ge 13$ (see Figure~\ref{fig:hex}), $\{0^+, 8^+, 10^+, 4^-, 6^-, 14^-\}$
is a splitting set for $H(n,\{0,1,5\})$ if $n \ge 16$, and  $\{ 0^+, 6^+, 10^+, 3^-, 7^-, 13^-\}$ is a splitting set for $H(n,\{0,2,5\})$ if $n \geq 16$.
\end{proof}

If $n < 13$ then conditions of Theorem~\ref{lem:tech} are not fulfilled. If $n = 12$ then $(n-1)^+ = 11^+$ which means that the
vertices of the graph in Figure~\ref{fig:hex} are not all distinct. If $n = 9$ then $9^- = 0^-$ since we work with $\mathbb{Z}_9$.
Similar arguments can be made if $n < 16$ in the case of the other two families from Corollary~\ref{thm:5}.
\begin{figure}
\centering
\usetikzlibrary{calc}
\begin{tikzpicture}
\tikzstyle{white} = [inner sep=1.5, draw, circle, fill=white]
\tikzstyle{black} = [inner sep=1.5, draw, circle, fill=black]
\tikzstyle{edge} = [draw, line width=0.8, color=black]
\node[white,label=-30:$7^-$] (7m) at (0, 0) {};
\node[black,label=30:$3^+$] (3p) at (30:1.5) {};
\node[white,label=-30:$4^-$] (4m) at ($ (3p) + (-30:1.5) $) {};
\node[black,label=30:$7^+$] (7p) at (-90:1.5) {};
\node[white,label=-30:$8^-$] (8m) at ($ (-90:1.5) + (-30:1.5) $) {};
\node[black,label=30:$4^+$] (4p) at ($ (4m) + (-90:1.5) $) {};
\draw[edge] (8m) -- (7p) -- (7m) -- (3p) -- (4m) -- (4p) -- (8m);
\node[white,label=-30:$3^-$] (3m) at ($ (3p) + (90:1.5) $) {};
\node[black,label=30:$0^+$] (0p) at ($ (4m) + (30:1.5) $) {};
\node[white,label=-30:$5^-$] (5m) at ($ (4p) + (-30:1.5) $) {};
\node[black,label=30:$8^+$] (8p) at ($ (8m) + (-90:1.5) $) {};
\node[white,label=-30:$11^-$] (11m) at ($ (7p) + (210:1.5) $) {};
\node[black,label=30:$6^+$] (6p) at ($ (7m) + (150:1.5) $) {};
\draw[edge] (11m) -- (7p);
\draw[edge] (3m) -- (3p);
\draw[edge] (4m) -- (0p);
\draw[edge] (5m) -- (4p);
\draw[edge] (8m) -- (8p);
\draw[edge] (7m) -- (6p);
\node[black,label=-180:$2^+$] (2p) at ($ (3m) + (150:1.5) $) {};
\node[black,label=0:$(n-1)^+$] (nm) at ($ (3m) + (30:1.5) $) {};
\draw[edge] (2p) -- (3m) -- (nm);
\node[white,label=0:$1^-$] (1m) at ($ (0p) + (90:1.5) $) {};
\node[white,label=0:$0^-$] (0m) at ($ (0p) + (-30:1.5) $) {};
\draw[edge] (1m) -- (0p) -- (0m);
\node[black,label=0:$1^+$] (1p) at ($ (5m) + (30:1.5) $) {};
\node[black,label=0:$5^+$] (5p) at ($ (5m) + (-90:1.5) $) {};
\draw[edge] (1p) -- (5m) -- (5p);
\node[white,label=0:$9^-$] (9m) at ($ (8p) + (-30:1.5) $) {};
\node[white,label=180:$12^-$] (12m) at ($ (8p) + (210:1.5) $) {};
\draw[edge] (12m) -- (8p) -- (9m);
\node[black,label=180:$11^+$] (11p) at ($ (11m) + (-90:1.5) $) {};
\node[black,label=180:$10^+$] (10p) at ($ (11m) + (150:1.5) $) {};
\draw[edge] (11p) -- (11m) -- (10p);
\node[white,label=180:$10^-$] (10m) at ($ (6p) + (90:1.5) $) {};
\node[white,label=180:$6^-$] (6m) at ($ (6p) + (210:1.5) $) {};
\draw[edge] (10m) -- (6p) -- (6m);
\end{tikzpicture}
\caption{The set $\Sigma = \{0^+,6^+,8^+,3^-,5^-,11^-\}$ is a splitting set for $H(n,\{0,1,4\})$ where $n \ge 13$.}
\label{fig:hex}
\end{figure}
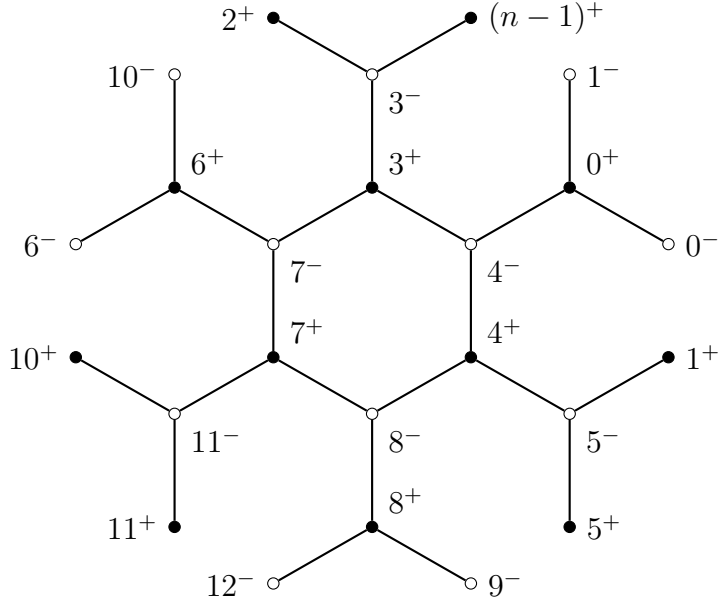

We investigated the first 100 graphs from the $H(n,\{0,1,4\})$ family. All but two are zero symmetric, nowadays called graphical
regular representation or GRR for short (see~\cite{CFP}). The exceptions are for $n = 13$ and $n = 15$.

By Corollary~\ref{thm:5}, there are infinitely many splittable $(n_3)$ configurations.
However, we are also able to show that there is no upper bound on the number of vertices of unsplittable $(n_3)$ configurations:

\begin{theorem}
\label{thm:infunsp}
There exist infinitely many cyclic $(n_3)$ configurations that are unsplittable.
\end{theorem}

\begin{proof}
We use the cyclic Haar graphs $X = H(n,\{0,1,3\})$, where $n \ge 7$. Clearly, each of them has girth 6. The graph can be written as $\mathrm{LCF}[5,-5]^n$. (For the LCF notation see~\cite{PS}.)  This means that the edges determined by symbols 0 and 1 form a Hamiltonian cycle while the edges arising from the symbol 3 form chords of length~5. See Figure~\ref{fig:3connSplittable} for an example.

Let us assume the result does not hold. This means there exists a splitting set $\Sigma$. By
removing $\Sigma$ from the graph the Hamiltonial cycle breaks into paths. Each path must contain at least two vertices. Let the sequence $\Pi = (p_1, p_2 ,\ldots, p_k)$ denote the lengths of the consecutive paths along the Hamiltonial cycle. The rest of the proof is in two steps:

\textbf{Step 1.}  If there are no two consecutive numbers of $\Pi$ equal to 2, then the corresponding segments are connected in $X - \Sigma$  since there is a chord of length 5 joining these two segments. But this means that all paths are connected by chords, so $\Sigma$ is not a
splitting set.

\textbf{Step 2.}  We can show that no two consecutive segments are of length 2. In case
of two adjacent segments of length 2 we would have vertices $\{i-3,i,i+3 \} \subseteq \Sigma$.  But that is
impossible, since $i-3$ is adjacent to $i+3$ in $X^2$.
\end{proof}

Note that this is not the only such family. Here is another one:

\begin{theorem}
Cyclic configurations defined by $H(3n,\{0,1,n\})$, where $n \geq 2$, are unsplittable.
\end{theorem}

\begin{proof}
The technique used here is similar to the technique used in proof of Theorem~\ref{thm:infunsp}. Let $X = H(3n,\{0,1,n\})$. The graph $X$ can be written as $\mathrm{LCF}[2n-1,-(2n-1) ]^{3n}$. Suppose that there exists a splitting set $\Sigma$. The edges determined by symbols 0 and 1 form a Hamiltonian cycle which breaks into paths when the splitting set $\Sigma$ is removed.

We show that any two consecutive paths are connected in $X - \Sigma$. Without loss of generality (because of symmetry), we may assume that $0^+ \in \Sigma$ is the vertex adjacent to the two paths under consideration. If $0^+ \in \Sigma$ then $1^-, 0^-, n^-, n^+, 1^+, 2n^+, (2n+1)^+ \notin \Sigma$. We show that vertices $1^-$ and $0^-$ (which belong to the two paths under consideration) are connected in $X - \Sigma$.

If $(2n+1)^- \notin \Sigma$ then $2n^+$ and $(2n+1)^+$ are connected in $X - \Sigma$. Since $0^-$ is adjacent to $2n^+$ and $1^-$ is adjacent to $(2n+1)^+$, vertices $0^-$ and $1^-$ are also connected in $X - \Sigma$. Now, suppose that $(2n+1)^- \in \Sigma$. This implies that $2n^-, (n+1)^+ (n+1)^- \notin \Sigma$. Then $2n^+$ is adjacent to $2n^-$, $2n^-$ is adjacent to $n^+$, $n^+$ is adjacent to $(n+1)^-$, $(n+1)^-$ is adjacent
to $1^+$, and $1^+$ is adjacent to $1^-$ in $X - \Sigma$. Therefore, $1^-$ and $0^-$ are connected in $X - \Sigma$.
\end{proof}

\begin{table}[htbp]
\centering
\caption{Overview of splittable and unsplittable connected cyclic Haar graphs.}
\label{tbl:numericsplit}

\vspace{0.5\baselineskip}
\begin{tabular}{c|c|c|c|c|c|c}
n & (a) & (b) & (c) & (d) & (e) & (f) \\
\hline
\hline
3 & 1 & 0 & 0 & 1 & 0 & 0 \\
4 & 1 & 0 & 0 & 1 & 0 & 0 \\
5 & 1 & 0 & 0 & 1 & 0 & 0 \\
6 & 2 & 0 & 0 & 2 & 0 & 0 \\
7 & 2 & 1 & 0 & 2 & 0 & 1 \\
8 & 3 & 1 & 1 & 2 & 0 & 1 \\
9 & 2 & 1 & 0 & 2 & 0 & 1 \\
10 & 3 & 1 & 1 & 2 & 0 & 1 \\
11 & 2 & 1 & 0 & 2 & 0 & 1 \\
12 & 5 & 3 & 1 & 4 & 0 & 3 \\
13 & 3 & 2 & 1 & 2 & 1 & 1 \\
14 & 4 & 2 & 2 & 2 & 1 & 1 \\
15 & 5 & 4 & 1 & 4 & 1 & 3 \\
16 & 5 & 3 & 3 & 2 & 2 & 1 \\
17 & 3 & 2 & 1 & 2 & 1 & 1 \\
18 & 6 & 4 & 3 & 3 & 2 & 2 \\
19 & 4 & 3 & 2 & 2 & 2 & 1 \\
20 & 7 & 5 & 5 & 2 & 4 & 1 \\
21 & 7 & 6 & 3 & 4 & 3 & 3 \\
22 & 6 & 4 & 4 & 2 & 3 & 1 \\
23 & 4 & 3 & 2 & 2 & 2 & 1 \\
24 & 11 & 9 & 7 & 4 & 6 & 3 \\
25 & 5 & 4 & 3 & 2 & 3 & 1 \\
26 & 7 & 5 & 5 & 2 & 4 & 1 \\
27 & 6 & 5 & 3 & 3 & 3 & 2 \\
28 & 9 & 7 & 7 & 2 & 6 & 1 \\
29 & 5 & 4 & 3 & 2 & 3 & 1 \\
30 & 13 & 11 & 9 & 4 & 8 & 3 \\
\end{tabular}

\vspace{0.5\baselineskip}
(a) Number of non-isomorphic connected cubic cyclic Haar graphs on $2n$ vertices.
(b) Those that have girth~6.
(c) Those that are splittable.
(d) Those that are unsplittable.
(e) Those that are splittable of girth 6.
(f) Those that are unsplittable of girth 6.
\end{table}

Cubic symmetric bicirculants were classified in \cite{dmtp2000} and \cite{tomo2007}. These results can be
summarised as follows:

\begin{theorem}[\cite{dmtp2000,tomo2007}]
A connected cubic symmetric graph is a bicirculant if and only if it is isomorphic to one of the following graphs:
\begin{enumerate}
\item the complete graph $K_4$,
\item the complete bipartite graph $K_{3,3}$,
\item the seven symmetric generalized Petersen graphs $GP(4, 1)$, $GP(5, 2)$, $GP(8, 3)$, $GP(10, 2)$, $GP(10, 3)$, $GP(12, 5)$ and $GP(24, 5)$,
\item the Heawood graph $H(7, \{0, 1, 3\})$, and
\item the cyclic Haar graph $H(n, \{0, 1, r+1\})$,  
where $n \geq 11$ is odd and $r \in \mathbb{Z}^*_n$ such
that $r^2 +r + 1 \equiv 0 \pmod n$. 
\end{enumerate}
\end{theorem}

It is well known that an $(n_3)$ configuration is flag-transitive if and only if its Levi graph is cubic symmetric graph of girth at least 6.
From Theorem 11 it follows that the girth of any connected cubic symmetric bicirculant is at most 6.
If the girth of such a graph is 6 or more then it is a Levi graph of a flag-transitive configuration. This enables us to characterise splittability of such configurations:

\begin{theorem}
\label{thm:main}
The Fano plane $(7_3)$, the Möbius-Kantor configuration $(8_3)$,  and the 
Desargues configuration $(10_3)$ are unsplittable. Their Levi graphs are $H(7, \{0, 1, 3\})$, $H(8, \{0, 1, 3\}) \cong GP(8, 3)$ and $GP(10, 3)$, respectively. 

If $n \geq 9$, all flag-transitive $(n_3)$ configurations, except the Desargues configuration, are splittable.
\end{theorem}

\begin{proof}
We start with the classification given in Theorem 11. Only bipartite graphs of girth 6 have to be considered. This rules out the complete graph $K_4$, the complete bipartite
graph $K_{3,3}$, and the generalised Petersen graphs $GP(5, 2)$, $GP (10, 2)$ and $GP(4, 1)$. Note that $GP(4, 1)$ is isomorphic to the cube graph $Q_3$.

It is well known, but one may check by computer that $GP(8, 3) \cong H(8, \{0, 1, 3\})$. See for instance \cite[Table 2]{HMP}.

One may also check by computer that $GP(8, 3)$,  $GP(10, 3)$ and the Heawood graph $H(7, \{0, 1, 3\})$ are unsplittable.

Let $V(GP(n, k)) = \{0, 1, \ldots, n - 1, 0', 1', \ldots, (n-1)'\}$ and $E(GP(n, k)) = \{\{i', ((i+1) \bmod n)'\}, \{i, i'\}, \{i, (i+k) \bmod n\} \mid  i = 0,\ldots, n - 1\}$.
The splitting set for $GP(12, 5)$ is $\Sigma = \{0', 4', 8', 2, 6, 10\}$ as shown in Figure~\ref{fig:splitting_12_5}. 
\begin{figure}[!htbp]
\centering
\begin{tikzpicture}[scale=1.45]
\tikzstyle{white} = [inner sep=1.8, draw, circle, fill=white]
\tikzstyle{red} = [inner sep=1.8, draw, circle, fill=magenta]
\tikzstyle{edge} = [draw, line width=0.6, color=black]
\node[white] (u0) at (90:1.5) {};
\node at (-0.2, 1.7) {$0$}; 
\node[white,label=90:$1$] (u1) at (120:1.5) {};
\node[red,label=90:$2$] (u2) at (150:1.5) {};
\node[white,label=90:$3$] (u3) at (180:1.5) {};
\node[white,label=120:$4$] (u4) at (210:1.5) {};
\node[white,label=180:$5$] (u5) at (240:1.5) {};
\node[red,label=-30:$6$] (u6) at (270:1.5) {};
\node[white,label=0:$7$] (u7) at (300:1.5) {};
\node[white,label=20:$8$] (u8) at (330:1.5) {};
\node[white,label=90:$9$] (u9) at (360:1.5) {};
\node[red,label=90:$10$] (u10) at (390:1.5) {};
\node[white] (u11) at (420:1.5) {};
\node at (0.6, 1.6) {$11$}; 
\node[red,label=90:$0'$] (v0) at (90:3) {};
\node[white,label=120:$1'$] (v1) at (120:3) {};
\node[white,label=150:$2'$] (v2) at (150:3) {};
\node[white,label=180:$3'$] (v3) at (180:3) {};
\node[red,label=210:$4'$] (v4) at (210:3) {};
\node[white,label=240:$5'$] (v5) at (240:3) {};
\node[white,label=270:$6'$] (v6) at (270:3) {};
\node[white,label=300:$7'$] (v7) at (300:3) {};
\node[red,label=330:$8'$] (v8) at (330:3) {};
\node[white,label=0:$9'$] (v9) at (360:3) {};
\node[white,label=30:$10'$] (v10) at (390:3) {};
\node[white,label=60:$11'$] (v11) at (420:3) {};
%
\foreach \i in {0,1,...,11} {
	\pgfmathtruncatemacro{\j}{mod(\i+1,12)}
	\pgfmathtruncatemacro{\k}{mod(\i+5,12)}
	\draw[edge] (v\i) -- (v\j);
	\draw[edge] (v\i) -- (u\i);
	\draw[edge] (u\i) -- (u\k);
}
\end{tikzpicture}
\caption{The splitting set for the Nauru graph $GP(12, 5)$~\cite{Eppstein2007,zhp2012}.}
\label{fig:splitting_12_5}
\end{figure}
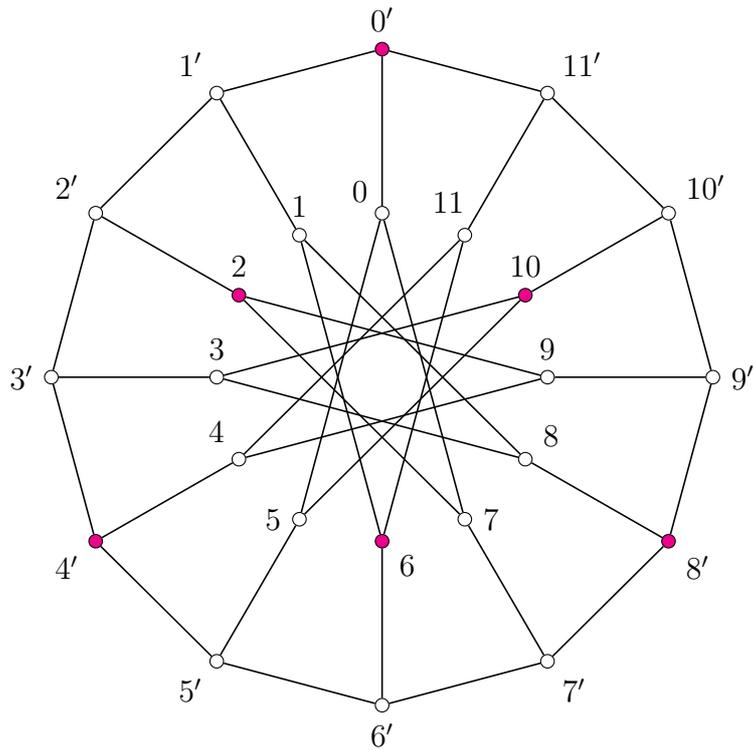
Note that $GP(12, 5) - S \cong 3C_6$, i.e., a disjoint union of three copies of $C_6$.
The splitting set for $GP(24, 5)$ is $\Sigma = \{ 0', 4', 8', 12', 16', 20', 2, 6, 10, 14, 18, 22 \}$ as shown in Figure~\ref{fig:splitting_24_5}.
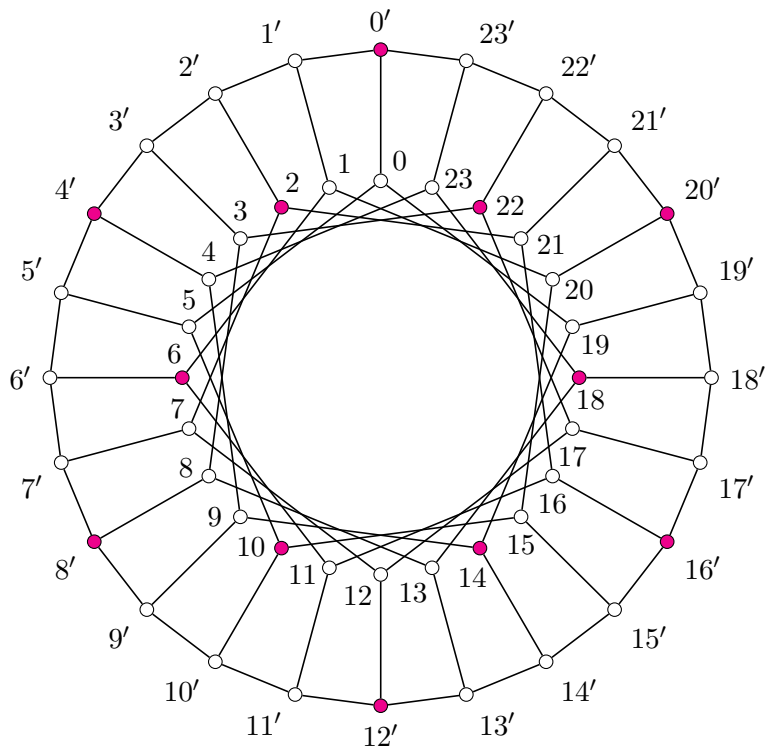
\begin{figure}[!htbp]
\centering
\begin{tikzpicture}[scale=1.45]
\tikzstyle{white} = [inner sep=1.8, draw, circle, fill=white]
\tikzstyle{red} = [inner sep=1.8, draw, circle, fill=magenta]
\tikzstyle{edge} = [draw, line width=0.6, color=black]
\node[white,label={[shift={(0.25,-0.1)}]\small $0$}] (u0) at (90:1.8) {};
\node[white,label={[shift={(0.2,-0.05)}]\small $1$}] (u1) at (105:1.8) {};
\node[red] (u2) at (120:1.8) {};
\node at (-0.8, 1.78) {\small $2$};
\node[white,label=90:{\small $3$}] (u3) at (135:1.8) {};
\node[white,label=90:{\small $4$}] (u4) at (150:1.8) {};
\node[white,label=90:{\small $5$}] (u5) at (165:1.8) {};
\node[red,label={[shift={(-0.1,-0.05)}]\small $6$}] (u6) at (180:1.8) {};
\node[white,label={[shift={(-0.15,-0.1)}]{\small $7$}}] (u7) at (195:1.8) {};
\node[white,label={[shift={(-0.3,-0.25)}]{\small $8$}}] (u8) at (210:1.8) {};
\node[white,label=180:{\small $9$}] (u9) at (225:1.8) {};
\node[red,label={[shift={(-0.4,-0.35)}]\small $10$}] (u10) at (240:1.8) {};
\node[white,label={[shift={(-0.35,-0.45)}]\small $11$}] (u11) at (255:1.8) {};
\node[white,label={[shift={(-0.3,-0.65)}]\small $12$}] (u12) at (270:1.8) {};
\node[white,label={[shift={(-0.25,-0.7)}]\small $13$}] (u13) at (285:1.8) {};
\node[red,label={[shift={(-0.1,-0.75)}]\small $14$}] (u14) at (300:1.8) {};
\node[white,label=-90:{\small $15$}] (u15) at (315:1.8) {};
\node[white,label=-90:{\small $16$}] (u16) at (330:1.8) {};
\node[white,label=-90:{\small $17$}] (u17) at (345:1.8) {};
\node[red,label={[shift={(0.15,-0.65)}]\small $18$}] (u18) at (360:1.8) {};
\node[white,label={[shift={(0.3,-0.6)}]\small $19$}] (u19) at (375:1.8) {};
\node[white,label={[shift={(0.35,-0.5)}]\small $20$}] (u20) at (390:1.8) {};
\node[white,label={[shift={(0.4,-0.45)}]\small $21$}] (u21) at (405:1.8) {};
\node[red,label={[shift={(0.4,-0.35)}]\small $22$}] (u22) at (420:1.8) {};
\node[white,label={[shift={(0.35,-0.25)}]\small $23$}] (u23) at (435:1.8) {};
\node[red,label=90:{\small $0'$}] (v0) at (90:3) {};
\node[white,label=105:{\small $1'$}] (v1) at (105:3) {};
\node[white,label=120:{\small $2'$}] (v2) at (120:3) {};
\node[white,label=135:{\small $3'$}] (v3) at (135:3) {};
\node[red,label=150:{\small $4'$}] (v4) at (150:3) {};
\node[white,label=165:{\small $5'$}] (v5) at (165:3) {};
\node[white,label=180:{\small $6'$}] (v6) at (180:3) {};
\node[white,label=195:{\small $7'$}] (v7) at (195:3) {};
\node[red,label=210:{\small $8'$}] (v8) at (210:3) {};
\node[white,label=225:{\small $9'$}] (v9) at (225:3) {};
\node[white,label=240:{\small $10'$}] (v10) at (240:3) {};
\node[white,label=255:{\small $11'$}] (v11) at (255:3) {};
\node[red,label=270:{\small $12'$}] (v12) at (270:3) {};
\node[white,label=285:{\small $13'$}] (v13) at (285:3) {};
\node[white,label=300:{\small $14'$}] (v14) at (300:3) {};
\node[white,label=315:{\small $15'$}] (v15) at (315:3) {};
\node[red,label=330:{\small $16'$}] (v16) at (330:3) {};
\node[white,label=345:{\small $17'$}] (v17) at (345:3) {};
\node[white,label=360:{\small $18'$}] (v18) at (360:3) {};
\node[white,label=375:{\small $19'$}] (v19) at (375:3) {};
\node[red,label=390:{\small $20'$}] (v20) at (390:3) {};
\node[white,label=405:{\small $21'$}] (v21) at (405:3) {};
\node[white,label=420:{\small $22'$}] (v22) at (420:3) {};
\node[white,label=435:{\small $23'$}] (v23) at (435:3) {};
%
%
\foreach \i in {0,1,...,23} {
	\pgfmathtruncatemacro{\j}{mod(\i+1,24)}
	\pgfmathtruncatemacro{\k}{mod(\i+5,24)}
	\draw[edge] (v\i) -- (v\j);
	\draw[edge] (v\i) -- (u\i);
	\draw[edge] (u\i) -- (u\k);
}
\end{tikzpicture}
\caption{The splitting set for $GP(24, 5)$ which was recently named the ADAM graph~\cite{adameditorial}.}
\label{fig:splitting_24_5}
\end{figure}
Note that $GP(24, 5) - S \cong 3C_{12}$. Also, note that $GP(24, 5)$ is not isomorphic to a cyclic Haar graph since its girth is $8$.

Using Theorem~\ref{lem:tech}, one may verify that all graphs in item 5 of Theorem 11 have girth 6 and for each of them the splitting set is $\{ 0^+, 2r^+, (2r+2)^+, r^-, (r+2)^-, (3r+2)^- \}$. We have
\begin{align*}
\mathcal{W} & = \{ 0, 1, r, r+1, r+2, 2r, 2r+1, 2r+2, 2r+3, 3r+1, 3r+2, 3r+3 \}, \\
\mathcal{B} & = \{ 0, 1, n - 1, r - 1, r, r+1, r+2,  2r, 2r+1, 2r+2, 3r+1, 3r+2  \}. 
\end{align*}
It is easy to verify that all elements of $\mathcal{W}$ are distinct and that all elements of $\mathcal{B}$ are distinct. For example,
suppose that $r \equiv 3r+3 \pmod n$. This means that
\begin{equation}
2r \equiv -3 \pmod n.
\label{eq1}
\end{equation}
From condition $r^2 +r + 1 \equiv 0 \pmod n$ we obtain
\begin{equation}
4r^2 + 4r + 4 = (2r)^2 + 2 \cdot 2r + 4 \equiv 0 \pmod n.
\label{eq2}
\end{equation}
Equations \eqref{eq1} and \eqref{eq2} together imply that $(-3)^2 + 2 \cdot (-3) + 4 = 7 \equiv 0 \pmod n$, which is a
contradiction since $n > 11$. All other cases can be checked in a similar way.
\end{proof}

From Theorem~\ref{thm:main} we directly obtain the following corollary.

\begin{corollary}
A cyclic flag-transitive $(n_3)$ configuration is splittable if and only if $n > 8$.
The only two exceptions are:
\begin{enumerate}
\item $H(7, \{0, 1, 3\})$, i.e.\ the Fano plane, and 
\item $H(8, \{0, 1, 3\})$, i.e.\ the Möbius-Kantor configuration.
\end{enumerate}
\end{corollary}

\begin{table}[!htbp]
\centering
\caption{List of non-isomorphic connected trivalent cyclic Haar graphs $H(n, S)$ with $n \leq 25$ and some of their properties.}
\label{default}

\vspace{0.5\baselineskip}
\footnotesize
\begin{tabular}{c|c|c|c|c|c}
$n$ & $S$ & (a) & (b) & (c) & (d) \\
\hline
\hline
$3$ & $\{0, 1, 2\}$ & $\bot$ & $4$ & $2$ & $\boldsymbol{\top}$ \\
$4$ & $\{0, 1, 2\}$ & $\bot$ & $4$ & $3$ & $\boldsymbol{\top}$ \\
$5$ & $\{0, 1, 2\}$ & $\bot$ & $4$ & $3$ & $\bot$ \\
$6$ & $\{0, 1, 2\}$ & $\bot$ & $4$ & $4$ & $\bot$ \\
$6$ & $\{0, 1, 3\}$ & $\bot$ & $4$ & $3$ & $\bot$ \\
$7$ & $\{0, 1, 2\}$ & $\bot$ & $4$ & $4$ & $\bot$ \\
$7$ & $\{0, 1, 3\}$ & $\bot$ & $6$ & $3$ & $\boldsymbol{\top}$ \\
$8$ & $\{0, 1, 2\}$ & $\bot$ & $4$ & $5$ & $\bot$ \\
$8$ & $\{0, 1, 3\}$ & $\bot$ & $6$ & $4$ & $\boldsymbol{\top}$ \\
$8$ & $\{0, 1, 4\}$ & $\boldsymbol{\top}$ & $4$ & $4$ & $\bot$ \\
$9$ & $\{0, 1, 2\}$ & $\bot$ & $4$ & $5$ & $\bot$ \\
$9$ & $\{0, 1, 3\}$ & $\bot$ & $6$ & $4$ & $\bot$ \\
$10$ & $\{0, 1, 2\}$ & $\bot$ & $4$ & $6$ & $\bot$ \\
$10$ & $\{0, 1, 3\}$ & $\bot$ & $6$ & $4$ & $\bot$ \\
$10$ & $\{0, 1, 5\}$ & $\boldsymbol{\top}$ & $4$ & $5$ & $\bot$ \\
$11$ & $\{0, 1, 2\}$ & $\bot$ & $4$ & $6$ & $\bot$ \\
$11$ & $\{0, 1, 3\}$ & $\bot$ & $6$ & $5$ & $\bot$ \\
$12$ & $\{0, 1, 2\}$ & $\bot$ & $4$ & $7$ & $\bot$ \\
$12$ & $\{0, 1, 3\}$ & $\bot$ & $6$ & $5$ & $\bot$ \\
$12$ & $\{0, 1, 4\}$ & $\bot$ & $6$ & $5$ & $\bot$ \\
$12$ & $\{0, 1, 5\}$ & $\bot$ & $6$ & $5$ & $\bot$ \\
$12$ & $\{0, 1, 6\}$ & $\boldsymbol{\top}$ & $4$ & $6$ & $\bot$ \\
$13$ & $\{0, 1, 2\}$ & $\bot$ & $4$ & $7$ & $\bot$ \\
$13$ & $\{0, 1, 3\}$ & $\bot$ & $6$ & $5$ & $\bot$ \\
$13$ & $\{0, 1, 4\}$ & $\boldsymbol{\top}$ & $6$ & $5$ & $\boldsymbol{\top}$ \\
$14$ & $\{0, 1, 2\}$ & $\bot$ & $4$ & $8$ & $\bot$ \\
$14$ & $\{0, 1, 3\}$ & $\bot$ & $6$ & $6$ & $\bot$ \\
$14$ & $\{0, 1, 4\}$ & $\boldsymbol{\top}$ & $6$ & $5$ & $\bot$ \\
$14$ & $\{0, 1, 7\}$ & $\boldsymbol{\top}$ & $4$ & $7$ & $\bot$ \\
$15$ & $\{0, 1, 2\}$ & $\bot$ & $4$ & $8$ & $\bot$ \\
$15$ & $\{0, 1, 3\}$ & $\bot$ & $6$ & $6$ & $\bot$ \\
$15$ & $\{0, 1, 4\}$ & $\boldsymbol{\top}$ & $6$ & $5$ & $\bot$ \\
$15$ & $\{0, 1, 5\}$ & $\bot$ & $6$ & $5$ & $\bot$ \\
$15$ & $\{0, 1, 6\}$ & $\bot$ & $6$ & $5$ & $\bot$ \\
$16$ & $\{0, 1, 2\}$ & $\bot$ & $4$ & $9$ & $\bot$ \\
$16$ & $\{0, 1, 3\}$ & $\bot$ & $6$ & $6$ & $\bot$ \\
$16$ & $\{0, 1, 4\}$ & $\boldsymbol{\top}$ & $6$ & $5$ & $\bot$ \\
$16$ & $\{0, 1, 7\}$ & $\boldsymbol{\top}$ & $6$ & $5$ & $\bot$ \\
$16$ & $\{0, 1, 8\}$ & $\boldsymbol{\top}$ & $4$ & $8$ & $\bot$ \\
$17$ & $\{0, 1, 2\}$ & $\bot$ & $4$ & $9$ & $\bot$ \\
$17$ & $\{0, 1, 3\}$ & $\bot$ & $6$ & $7$ & $\bot$ \\
$17$ & $\{0, 1, 4\}$ & $\boldsymbol{\top}$ & $6$ & $5$ & $\bot$ \\
$18$ & $\{0, 1, 2\}$ & $\bot$ & $4$ & $10$ & $\bot$ \\
$18$ & $\{0, 1, 3\}$ & $\bot$ & $6$ & $7$ & $\bot$ \\
$18$ & $\{0, 1, 4\}$ & $\boldsymbol{\top}$ & $6$ & $6$ & $\bot$ \\
$18$ & $\{0, 1, 5\}$ & $\boldsymbol{\top}$ & $6$ & $6$ & $\bot$
\end{tabular}
\qquad
\begin{tabular}{c|c|c|c|c|c}
$n$ & $S$ & (a) & (b) & (c) & (d) \\
\hline
\hline
$18$ & $\{0, 1, 6\}$ & $\bot$ & $6$ & $6$ & $\bot$ \\
$18$ & $\{0, 1, 9\}$ & $\boldsymbol{\top}$ & $4$ & $9$ & $\bot$ \\
$19$ & $\{0, 1, 2\}$ & $\bot$ & $4$ & $10$ & $\bot$ \\
$19$ & $\{0, 1, 3\}$ & $\bot$ & $6$ & $7$ & $\bot$ \\
$19$ & $\{0, 1, 4\}$ & $\boldsymbol{\top}$ & $6$ & $6$ & $\bot$ \\
$19$ & $\{0, 1, 8\}$ & $\boldsymbol{\top}$ & $6$ & $5$ & $\boldsymbol{\top}$ \\
$20$ & $\{0, 1, 2\}$ & $\bot$ & $4$ & $11$ & $\bot$ \\
$20$ & $\{0, 1, 3\}$ & $\bot$ & $6$ & $8$ & $\bot$ \\
$20$ & $\{0, 1, 4\}$ & $\boldsymbol{\top}$ & $6$ & $6$ & $\bot$ \\
$20$ & $\{0, 1, 5\}$ & $\boldsymbol{\top}$ & $6$ & $6$ & $\bot$ \\
$20$ & $\{0, 1, 6\}$ & $\boldsymbol{\top}$ & $6$ & $6$ & $\bot$ \\
$20$ & $\{0, 1, 9\}$ & $\boldsymbol{\top}$ & $6$ & $7$ & $\bot$ \\
$20$ & $\{0, 1, 10\}$ & $\boldsymbol{\top}$ & $4$ & $10$ & $\bot$ \\
$21$ & $\{0, 1, 2\}$ & $\bot$ & $4$ & $11$ & $\bot$ \\
$21$ & $\{0, 1, 3\}$ & $\bot$ & $6$ & $8$ & $\bot$ \\
$21$ & $\{0, 1, 4\}$ & $\boldsymbol{\top}$ & $6$ & $6$ & $\bot$ \\
$21$ & $\{0, 1, 5\}$ & $\boldsymbol{\top}$ & $6$ & $6$ & $\boldsymbol{\top}$ \\
$21$ & $\{0, 1, 7\}$ & $\bot$ & $6$ & $7$ & $\bot$ \\
$21$ & $\{0, 1, 8\}$ & $\bot$ & $6$ & $7$ & $\bot$ \\
$21$ & $\{0, 1, 9\}$ & $\boldsymbol{\top}$ & $6$ & $6$ & $\bot$ \\
$22$ & $\{0, 1, 2\}$ & $\bot$ & $4$ & $12$ & $\bot$ \\
$22$ & $\{0, 1, 3\}$ & $\bot$ & $6$ & $8$ & $\bot$ \\
$22$ & $\{0, 1, 4\}$ & $\boldsymbol{\top}$ & $6$ & $7$ & $\bot$ \\
$22$ & $\{0, 1, 5\}$ & $\boldsymbol{\top}$ & $6$ & $6$ & $\bot$ \\
$22$ & $\{0, 1, 6\}$ & $\boldsymbol{\top}$ & $6$ & $7$ & $\bot$ \\
$22$ & $\{0, 1, 11\}$ & $\boldsymbol{\top}$ & $4$ & $11$ & $\bot$ \\
$23$ & $\{0, 1, 2\}$ & $\bot$ & $4$ & $12$ & $\bot$ \\
$23$ & $\{0, 1, 3\}$ & $\bot$ & $6$ & $9$ & $\bot$ \\
$23$ & $\{0, 1, 4\}$ & $\boldsymbol{\top}$ & $6$ & $7$ & $\bot$ \\
$23$ & $\{0, 1, 5\}$ & $\boldsymbol{\top}$ & $6$ & $7$ & $\bot$ \\
$24$ & $\{0, 1, 2\}$ & $\bot$ & $4$ & $13$ & $\bot$ \\
$24$ & $\{0, 1, 3\}$ & $\bot$ & $6$ & $9$ & $\bot$ \\
$24$ & $\{0, 1, 4\}$ & $\boldsymbol{\top}$ & $6$ & $7$ & $\bot$ \\
$24$ & $\{0, 1, 5\}$ & $\boldsymbol{\top}$ & $6$ & $7$ & $\bot$ \\
$24$ & $\{0, 1, 6\}$ & $\boldsymbol{\top}$ & $6$ & $7$ & $\bot$ \\
$24$ & $\{0, 1, 7\}$ & $\boldsymbol{\top}$ & $6$ & $7$ & $\bot$ \\
$24$ & $\{0, 1, 8\}$ & $\bot$ & $6$ & $8$ & $\bot$ \\
$24$ & $\{0, 1, 9\}$ & $\bot$ & $6$ & $8$ & $\bot$ \\
$24$ & $\{0, 1, 10\}$ & $\boldsymbol{\top}$ & $6$ & $6$ & $\bot$ \\
$24$ & $\{0, 1, 11\}$ & $\boldsymbol{\top}$ & $6$ & $7$ & $\bot$ \\
$24$ & $\{0, 1, 12\}$ & $\boldsymbol{\top}$ & $4$ & $12$ & $\bot$ \\
$25$ & $\{0, 1, 2\}$ & $\bot$ & $4$ & $13$ & $\bot$ \\
$25$ & $\{0, 1, 3\}$ & $\bot$ & $6$ & $9$ & $\bot$ \\
$25$ & $\{0, 1, 4\}$ & $\boldsymbol{\top}$ & $6$ & $7$ & $\bot$ \\
$25$ & $\{0, 1, 5\}$ & $\boldsymbol{\top}$ & $6$ & $7$ & $\bot$ \\
$25$ & $\{0, 1, 10\}$ & $\boldsymbol{\top}$ & $6$ & $7$ & $\bot$ 
\end{tabular}

\vspace{0.5\baselineskip}
\small
(a) splittable?
(b) girth
(c) diameter
(d) arc-transitive?
\end{table}%

\begin{table}[!htbp]
\centering
\caption{List of non-isomorphic connected trivalent cyclic Haar graphs $H(n, S)$ with $26 \leq n \leq 30$ and some of their properties.}
\label{default22}

\vspace{0.5\baselineskip}
\small
\begin{tabular}{c|c|c|c|c|c}
$n$ & $S$ & (a) & (b) & (c) & (d) \\
\hline
\hline
$26$ & $\{0, 1, 2\}$ & $\bot$ & $4$ & $14$ & $\bot$ \\
$26$ & $\{0, 1, 3\}$ & $\bot$ & $6$ & $10$ & $\bot$ \\
$26$ & $\{0, 1, 4\}$ & $\boldsymbol{\top}$ & $6$ & $8$ & $\bot$ \\
$26$ & $\{0, 1, 5\}$ & $\boldsymbol{\top}$ & $6$ & $7$ & $\bot$ \\
$26$ & $\{0, 1, 7\}$ & $\boldsymbol{\top}$ & $6$ & $8$ & $\bot$ \\
$26$ & $\{0, 1, 8\}$ & $\boldsymbol{\top}$ & $6$ & $7$ & $\bot$ \\
$26$ & $\{0, 1, 13\}$ & $\boldsymbol{\top}$ & $4$ & $13$ & $\bot$ \\
$27$ & $\{0, 1, 2\}$ & $\bot$ & $4$ & $14$ & $\bot$ \\
$27$ & $\{0, 1, 3\}$ & $\bot$ & $6$ & $10$ & $\bot$ \\
$27$ & $\{0, 1, 4\}$ & $\boldsymbol{\top}$ & $6$ & $8$ & $\bot$ \\
$27$ & $\{0, 1, 5\}$ & $\boldsymbol{\top}$ & $6$ & $7$ & $\bot$ \\
$27$ & $\{0, 1, 6\}$ & $\boldsymbol{\top}$ & $6$ & $7$ & $\bot$ \\
$27$ & $\{0, 1, 9\}$ & $\bot$ & $6$ & $9$ & $\bot$ \\
$28$ & $\{0, 1, 2\}$ & $\bot$ & $4$ & $15$ & $\bot$ \\
$28$ & $\{0, 1, 3\}$ & $\bot$ & $6$ & $10$ & $\bot$ \\
$28$ & $\{0, 1, 4\}$ & $\boldsymbol{\top}$ & $6$ & $8$ & $\bot$ \\
$28$ & $\{0, 1, 5\}$ & $\boldsymbol{\top}$ & $6$ & $7$ & $\bot$ \\
$28$ & $\{0, 1, 6\}$ & $\boldsymbol{\top}$ & $6$ & $8$ & $\bot$ \\
$28$ & $\{0, 1, 7\}$ & $\boldsymbol{\top}$ & $6$ & $7$ & $\bot$ \\
$28$ & $\{0, 1, 8\}$ & $\boldsymbol{\top}$ & $6$ & $7$ & $\bot$ \\
$28$ & $\{0, 1, 13\}$ & $\boldsymbol{\top}$ & $6$ & $9$ & $\bot$ \\
$28$ & $\{0, 1, 14\}$ & $\boldsymbol{\top}$ & $4$ & $14$ & $\bot$ \\
$29$ & $\{0, 1, 2\}$ & $\bot$ & $4$ & $15$ & $\bot$ \\
$29$ & $\{0, 1, 3\}$ & $\bot$ & $6$ & $11$ & $\bot$ \\
$29$ & $\{0, 1, 4\}$ & $\boldsymbol{\top}$ & $6$ & $8$ & $\bot$ \\
$29$ & $\{0, 1, 5\}$ & $\boldsymbol{\top}$ & $6$ & $7$ & $\bot$ \\
$29$ & $\{0, 1, 9\}$ & $\boldsymbol{\top}$ & $6$ & $7$ & $\bot$ \\
$30$ & $\{0, 1, 2\}$ & $\bot$ & $4$ & $16$ & $\bot$ \\
$30$ & $\{0, 1, 3\}$ & $\bot$ & $6$ & $11$ & $\bot$ \\
$30$ & $\{0, 1, 4\}$ & $\boldsymbol{\top}$ & $6$ & $9$ & $\bot$ \\
$30$ & $\{0, 1, 5\}$ & $\boldsymbol{\top}$ & $6$ & $7$ & $\bot$ \\
$30$ & $\{0, 1, 6\}$ & $\boldsymbol{\top}$ & $6$ & $7$ & $\bot$ \\
$30$ & $\{0, 1, 7\}$ & $\boldsymbol{\top}$ & $6$ & $7$ & $\bot$ \\
$30$ & $\{0, 1, 8\}$ & $\boldsymbol{\top}$ & $6$ & $9$ & $\bot$ \\
$30$ & $\{0, 1, 9\}$ & $\boldsymbol{\top}$ & $6$ & $7$ & $\bot$ \\
$30$ & $\{0, 1, 10\}$ & $\bot$ & $6$ & $10$ & $\bot$ \\
$30$ & $\{0, 1, 11\}$ & $\bot$ & $6$ & $10$ & $\bot$ \\
$30$ & $\{0, 1, 12\}$ & $\boldsymbol{\top}$ & $6$ & $8$ & $\bot$ \\
$30$ & $\{0, 1, 15\}$ & $\boldsymbol{\top}$ & $4$ & $15$ & $\bot$ \\
$30$ & $\{0, 2, 5\}$ & $\boldsymbol{\top}$ & $6$ & $8$ & $\bot$ \\
\end{tabular}

\vspace{0.5\baselineskip}
(a) splittable?
(b) girth
(c) diameter
(d) arc-transitive?
\end{table}%

\section{Splittable geometric $\boldsymbol{(n_k)}$ configurations}
\label{sec:sgc}
We will now show that for any $k$ there exist a geometric, triangle-free, $(n_k)$ configuration which is of type T1, i.e., it is point-splittable
and line-splittable.

Let us first provide a construction to obtain a geometric $(n_k)$ configuration for any $k$. We start with an unbalanced
$(k_1, 1_k)$ configuration, denoted $\mathcal{G}_k^{(1)}$, that consists of a single line containing $k$ points. Let $\mathcal{G}_k^{(i)}$ be
a configuration that is obtained from $\mathcal{G}_k^{(i-1)}$ by the \emph{$k$-fold parallel replication} (see \cite[p.\ 245]{PS}).
The configuration $\mathcal{G}_k^{(k)}$ is a balanced $(k^k{}_k)$ configuration, called a \emph{generalised Gray configuration}.

\begin{lemma}
\label{lem:constr}
Let $\mathcal{C}$ be an arbitrary geometric $(n_k)$ configuration. There exists a geometric $(kn_k)$ configuration
$\mathcal{D}$ that is point- and line-splittable. Moreover, if $\mathcal{C}$ is triangle-free then $\mathcal{D}$ is also triangle-free.
\end{lemma}

\begin{proof}
Let $\mathcal{C}$ be as stated. Select an arbitrary line $L$ of $\mathcal{C}$ passing through points
$p^{(1)}, p^{(2)}, \ldots, p^{(k)}$ of $\mathcal{C}$ as shown in Figure~\ref{fig:constr}~(a).
Remove line $L$ and call the resulting structure $\mathcal{C}'$. Make $k$ copies of $\mathcal{C}'$: $\mathcal{C}'_1, \mathcal{C}'_2, \ldots, \mathcal{C}'_k$ and
place them equally spaced in any direction $\vec{v}$ that is non-parallel to the direction of any line of $\mathcal{C}'$ (see Figure~\ref{fig:constr}~(b)).
\begin{figure}[!htbp]
\centering
\subfigure[]{
\usetikzlibrary{calc}
\begin{tikzpicture}[scale=0.8]
\definecolor{pomaranca}{rgb}{.80,0.0,.80}
\tikzstyle{every node}=[draw,circle,fill,inner sep=1.2pt]
\tikzstyle{lipsa}=[color=dgc,line width=1pt,fill=lgc]
\definecolor{dgc}{rgb}{0.0,0.300,0.0}
\definecolor{lgc}{rgb}{0.950,1.0,0.950}
\draw[rotate around={-30:(5.2,5)},lipsa] (5.2,5) ellipse (35pt and 60pt);
\begin{footnotesize}
\coordinate (l1) at ($ (5, 5) + (60:2.7) $);
\coordinate (l2) at ($ (5, 5) - (60:2.5) $);
\draw[color=pomaranca,line width=1.2pt] (l1) -- (l2);
\node[label=0:$p^{(2)}$] (p2)  at  ($ (5, 5) + (60:0.4) $) {};
\node[label=0:$p^{(1)}$] (p1)  at  ($ (5, 5) + (60:1.2) $) {};
\node[label=0:$p^{(k)}$] (pk)  at  ($ (5, 5) - (60:1.2) $) {};
\node[draw=none,fill=none] (plbl)  at  ($ (5.3, 5) - (60:0.4) $) {$\ldots$};
\node[draw=none,fill=none,color=pomaranca] (plbl)  at  ($ (5.3, 5) + (60:2.5) $) {$L$};
\node[draw=none,fill=none] (plbl)  at  ($ (6.7, 5.2) - (60:2.5) $) {$\mathcal{C}$};
\end{footnotesize}
\end{tikzpicture}
}
\subfigure[]{
\usetikzlibrary{calc}
\begin{tikzpicture}[scale=0.8]
\definecolor{pomaranca}{rgb}{.80,0.0,.80}
\tikzstyle{every node}=[draw,circle,fill,inner sep=1.2pt]
\tikzstyle{lipsa}=[color=dgc,line width=1pt,fill=lgc]
\definecolor{dgc}{rgb}{0.0,0.300,0.0}
\definecolor{lgc}{rgb}{0.950,1.0,0.950}
\draw[rotate around={-30:(5.2,5)},lipsa] (5.2,5) ellipse (35pt and 60pt);
\draw[rotate around={-30:(8.7,5)},lipsa] (8.7,5) ellipse (35pt and 60pt);
\draw[rotate around={-30:(13.7,5)},lipsa] (13.7,5) ellipse (35pt and 60pt);
\begin{footnotesize}
\coordinate (m1l)  at  ($ (5, 5) + (60:1.2) - (1.5, 0)$);
\coordinate (m1d)  at  ($ (13.5, 5) + (60:1.2) + (1.4, 0)$);
\draw[color=pomaranca,line width=1.2pt] (m1l) -- (m1d);
\coordinate (m2l)  at  ($ (5, 5) + (60:0.4) - (1.5, 0)$);
\coordinate (m2d)  at  ($ (13.5, 5) + (60:0.4) + (1.7, 0)$);
\draw[color=pomaranca,line width=1.2pt] (m2l) -- (m2d);
\coordinate (mkl)  at  ($ (5, 5) - (60:1.2) - (1.0, 0)$);
\coordinate (mkd)  at  ($ (13.5, 5) - (60:1.2) + (1.9, 0)$);
\draw[color=pomaranca,line width=1.2pt] (mkl) -- (mkd);
\node[label=-30:$p^{(2)}_1$] (p2)  at  ($ (5, 5) + (60:0.4) $) {};
\node[label={[yshift=1.5]-30:$p^{(1)}_1$}] (p1)  at  ($ (5, 5) + (60:1.2) $) {};
\node[label=-30:$p^{(k)}_1$] (pk)  at  ($ (5, 5) - (60:1.2) $) {};
\node[label=-30:$p^{(2)}_2$] (p2a)  at  ($ (8.5, 5) + (60:0.4) $) {};
\node[label={[yshift=1.5]-30:$p^{(1)}_2$}] (p1a)  at  ($ (8.5, 5) + (60:1.2) $) {};
\node[label=-30:$p^{(k)}_2$] (pka)  at  ($ (8.5, 5) - (60:1.2) $) {};
\node[label=-30:$p^{(2)}_k$] (p2b)  at  ($ (13.5, 5) + (60:0.4) $) {};
\node[label={[yshift=1.5]-30:$p^{(1)}_k$}] (p1b)  at  ($ (13.5, 5) + (60:1.2) $) {};
\node[label=-30:$p^{(k)}_k$] (pkb)  at  ($ (13.5, 5) - (60:1.2) $) {};
\node[draw=none,fill=none] (plbl)  at  ($ (5.1, 5) - (60:0.4) $) {$\ldots$};
\node[draw=none,fill=none] (plbl)  at  ($ (8.6, 5) - (60:0.4) $) {$\ldots$};
\node[draw=none,fill=none] (plbl)  at  ($ (13.6, 5) - (60:0.4) $) {$\ldots$};
\node[draw=none,fill=none] (plbl)  at  (11.1, 4.5) {$\cdots$};
\node[draw=none,fill=none,color=pomaranca] (plbl)  at  ($ (m1d) + (0.4,0) $) {$M_1$};
\node[draw=none,fill=none,color=pomaranca] (plbl)  at  ($ (m2d) + (0.4,0) $) {$M_2$};
\node[draw=none,fill=none,color=pomaranca] (plbl)  at  ($ (mkd) + (0.4,0) $) {$M_k$};
\node[draw=none,fill=none] (plbl)  at  ($ (6.7, 5.2) - (60:2.5) $) {$\mathcal{C}_1'$};
\node[draw=none,fill=none] (plbl)  at  ($ (10.2, 5.2) - (60:2.5) $) {$\mathcal{C}_2'$};
\node[draw=none,fill=none] (plbl)  at  ($ (15.2, 5.2) - (60:2.5) $) {$\mathcal{C}_k'$};
\end{footnotesize}
\end{tikzpicture}
}
\caption{Construction provided by Lemma~\ref{lem:constr}.}
\label{fig:constr}
\end{figure}
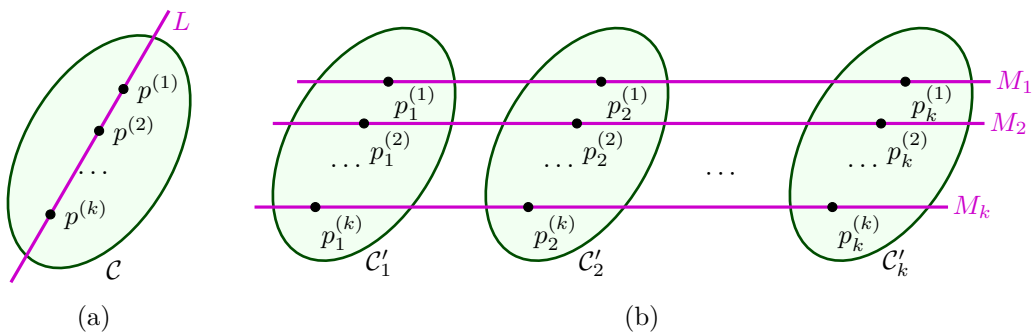
Point of $\mathcal{C}'_i$ that correspond to $p^{(j)}$ in $\mathcal{C}'$ is denoted $p_i^{(j)}$. Now add lines $M_1, M_2, \ldots, M_k$,
such that $M_i$ passes through points $p_1^{(i)}, p_2^{(i)}, \ldots, p_k^{(i)}$. The resulting structure, denoted $\mathcal{D}$,
is clearly a $(kn_k)$ configuration.

The set of lines $\{M_1, M_2, \ldots, M_k\}$ is a splitting set of $\mathcal{D}$ which proves that $\mathcal{D}$ is line-splittable. The set of points $\{p_i^{(1)}, p_i^{(2)}, \ldots, p_i^{(k)}\}$ is a splitting set for an arbitrary $1 \leq i \leq k$ which proves that $\mathcal{D}$ is also point-splittable.

It is easy to see that the resulting structure $\mathcal{D}$ is triangle-free.
\end{proof}

Now we can state the main result of this section.

\begin{theorem}
For any $k >1$ and any $n_0$ there exist a number $n > n_0$, such that there exists a splittable $(n_k)$ configuration.
\end{theorem}

\begin{proof}
Let $\mathcal{C}_0 = \mathcal{G}_k^{(k)}$, i.e.\ the generalised Gray $(k^k{}_k)$ configuration. Let $\mathcal{C}_i$ be a configuration obtained from $\mathcal{C}_{i-1}$ by an application of Lemma~14. Note that the obtained configuration $\mathcal{C}_i$ is not
uniquely defined -- it depends on the choice of the line~$L$.

From Lemma~14 it follows that each $\mathcal{C}_i$, $i \geq 1$, is a point- and
line-splittable configuration. Each configuration $\mathcal{C}_i$ is balanced and the
number of points of $\mathcal{C}_{i+1}$ is strictly greater than the number of points of $\mathcal{C}_i$. Therefore, for increasing
values of $i$, the number of points will eventually exceed any given number~$n_0$.
\end{proof}

Since configurations $\mathcal{C}_1, \mathcal{C}_2, \ldots$ constructed in the proof of Theorem 15 are all of type T1, their
duals are also of type T1.

\begin{example}
The generalised Gray $(k^k{}_k)$ configuration for $k = 3$ is simply called the \emph{Gray configuration} (see Figure~\ref{fig:gray}(a)). 
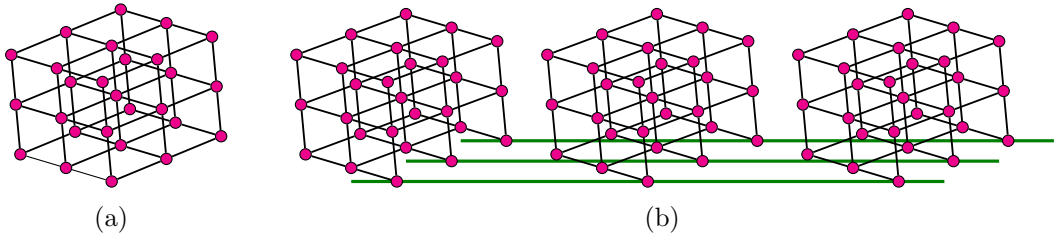
\begin{figure}[!htbp]
\centering
\subfigure[]{
\usetikzlibrary{calc}
\begin{tikzpicture}[scale=0.6]
\tikzstyle{every node}=[draw,circle,inner sep=1.5pt,fill=magenta]
\tikzstyle{edge}=[line width=0.7]
\coordinate (v1) at (1, -0.3);
\coordinate (v2) at (1.2, 0.5);
\coordinate (v3) at (-0.1, 1.1);
\node (a111) at (0, 0) {};
\node (a211) at ($ (0, 0) + (v1) $) {};
\node (a311) at ($ (0, 0) + 2*(v1) $) {};
\node (a121) at ($ (0, 0) + (v2) $) {};
\node (a221) at ($ (0, 0) + (v1) + (v2) $) {};
\node (a321) at ($ (0, 0) + 2*(v1) + (v2) $) {};
\node (a131) at ($ (0, 0) + 2*(v2) $) {};
\node (a231) at ($ (0, 0) + (v1) + 2*(v2) $) {};
\node (a331) at ($ (0, 0) + 2*(v1) + 2*(v2) $) {};
\draw (a111) -- (a211) -- (a311);
\draw [edge] (a121) -- (a221) -- (a321);
\draw [edge] (a131) -- (a231) -- (a331);
\draw [edge] (a111) -- (a121) -- (a131);
\draw [edge] (a211) -- (a221) -- (a231);
\draw [edge] (a311) -- (a321) -- (a331);
\node (a112) at ($ (0, 0) + (v3) $) {};
\node (a212) at ($ (0, 0) + (v1) + (v3) $) {};
\node (a312) at ($ (0, 0) + 2*(v1) + (v3) $) {};
\node (a122) at ($ (0, 0) + (v2) + (v3) $) {};
\node (a222) at ($ (0, 0) + (v1) + (v2) + (v3) $) {};
\node (a322) at ($ (0, 0) + 2*(v1) + (v2) + (v3) $) {};
\node (a132) at ($ (0, 0) + 2*(v2) + (v3) $) {};
\node (a232) at ($ (0, 0) + (v1) + 2*(v2) + (v3) $) {};
\node (a332) at ($ (0, 0) + 2*(v1) + 2*(v2) + (v3) $) {};
\draw [edge] (a112) -- (a212) -- (a312);
\draw [edge] (a122) -- (a222) -- (a322);
\draw [edge] (a132) -- (a232) -- (a332);
\draw [edge] (a112) -- (a122) -- (a132);
\draw [edge] (a212) -- (a222) -- (a232);
\draw [edge] (a312) -- (a322) -- (a332);
\node (a113) at ($ (0, 0) + 2*(v3) $) {};
\node (a213) at ($ (0, 0) + (v1) + 2*(v3) $) {};
\node (a313) at ($ (0, 0) + 2*(v1) + 2*(v3) $) {};
\node (a123) at ($ (0, 0) + (v2) + 2*(v3) $) {};
\node (a223) at ($ (0, 0) + (v1) + (v2) + 2*(v3) $) {};
\node (a323) at ($ (0, 0) + 2*(v1) + (v2) + 2*(v3) $) {};
\node (a133) at ($ (0, 0) + 2*(v2) + 2*(v3) $) {};
\node (a233) at ($ (0, 0) + (v1) + 2*(v2) + 2*(v3) $) {};
\node (a333) at ($ (0, 0) + 2*(v1) + 2*(v2) + 2*(v3) $) {};
\draw [edge] (a113) -- (a213) -- (a313);
\draw [edge] (a123) -- (a223) -- (a323);
\draw [edge] (a133) -- (a233) -- (a333);
\draw [edge] (a113) -- (a123) -- (a133);
\draw [edge] (a213) -- (a223) -- (a233);
\draw [edge] (a313) -- (a323) -- (a333);
\draw [edge] (a111) -- (a112) -- (a113);
\draw [edge] (a211) -- (a212) -- (a213);
\draw [edge] (a311) -- (a312) -- (a313);
\draw [edge] (a121) -- (a122) -- (a123);
\draw [edge] (a221) -- (a222) -- (a223);
\draw [edge] (a321) -- (a322) -- (a323);
\draw [edge] (a131) -- (a132) -- (a133);
\draw [edge] (a231) -- (a232) -- (a233);
\draw [edge] (a331) -- (a332) -- (a333);
\end{tikzpicture}
}
\subfigure[]{
\usetikzlibrary{calc}
\definecolor{dgrin}{rgb}{0.0,0.5,0.0}
\begin{tikzpicture}[scale=0.6]
\tikzstyle{every node}=[draw,circle,inner sep=1.5pt,fill=magenta]
\tikzstyle{edge}=[line width=0.7]
\tikzstyle{special edge}=[line width=1.2,color=dgrin]
\coordinate (v1) at (1, -0.3);
\coordinate (v2) at (1.2, 0.45); 
\coordinate (v3) at (-0.1, 1.1);
\coordinate (k3) at (11, 0);
\coordinate (ka311) at ($ (0, 0) + 2*(v1) $);
\coordinate (kc311) at ($ (k3) + 2*(v1) $);
\coordinate (ka321) at ($ (0, 0) + 2*(v1) + (v2) $);
\coordinate (kc321) at ($ (k3) + 2*(v1) + (v2) $);
\coordinate (ka331) at ($ (0, 0) + 2*(v1) + 2*(v2) $);
\coordinate (kc331) at ($ (k3) + 2*(v1) + 2*(v2) $);
\draw [special edge] ($ (ka311) - (1,0) $) -- ($ (kc311) + (1,0) $);
\draw [special edge] ($ (ka321) - (1,0) $) -- ($ (kc321) + (1,0) $);
\draw [special edge] ($ (ka331) - (1,0) $) -- ($ (kc331) + (1,0) $);
\node (a311) at ($ (0, 0) + 2*(v1) $) {};
\node (c311) at ($ (k3) + 2*(v1) $) {};
\node (a321) at ($ (0, 0) + 2*(v1) + (v2) $) {};
\node (c321) at ($ (k3) + 2*(v1) + (v2) $) {};
\node (a331) at ($ (0, 0) + 2*(v1) + 2*(v2) $) {};
\node (c331) at ($ (k3) + 2*(v1) + 2*(v2) $) {};
\node (a111) at (0, 0) {};
\node (a211) at ($ (0, 0) + (v1) $) {};
\node (a121) at ($ (0, 0) + (v2) $) {};
\node (a221) at ($ (0, 0) + (v1) + (v2) $) {};
\node (a131) at ($ (0, 0) + 2*(v2) $) {};
\node (a231) at ($ (0, 0) + (v1) + 2*(v2) $) {};
\node (a112) at ($ (0, 0) + (v3) $) {};
\node (a212) at ($ (0, 0) + (v1) + (v3) $) {};
\node (a312) at ($ (0, 0) + 2*(v1) + (v3) $) {};
\node (a122) at ($ (0, 0) + (v2) + (v3) $) {};
\node (a222) at ($ (0, 0) + (v1) + (v2) + (v3) $) {};
\node (a322) at ($ (0, 0) + 2*(v1) + (v2) + (v3) $) {};
\node (a132) at ($ (0, 0) + 2*(v2) + (v3) $) {};
\node (a232) at ($ (0, 0) + (v1) + 2*(v2) + (v3) $) {};
\node (a332) at ($ (0, 0) + 2*(v1) + 2*(v2) + (v3) $) {};
\node (a113) at ($ (0, 0) + 2*(v3) $) {};
\node (a213) at ($ (0, 0) + (v1) + 2*(v3) $) {};
\node (a313) at ($ (0, 0) + 2*(v1) + 2*(v3) $) {};
\node (a123) at ($ (0, 0) + (v2) + 2*(v3) $) {};
\node (a223) at ($ (0, 0) + (v1) + (v2) + 2*(v3) $) {};
\node (a323) at ($ (0, 0) + 2*(v1) + (v2) + 2*(v3) $) {};
\node (a133) at ($ (0, 0) + 2*(v2) + 2*(v3) $) {};
\node (a233) at ($ (0, 0) + (v1) + 2*(v2) + 2*(v3) $) {};
\node (a333) at ($ (0, 0) + 2*(v1) + 2*(v2) + 2*(v3) $) {};
\coordinate (k2) at (5.5, 0);
\node (b111) at (k2) {};
\node (b211) at ($ (k2) + (v1) $) {};
\node (b311) at ($ (k2) + 2*(v1) $) {};
\node (b121) at ($ (k2) + (v2) $) {};
\node (b221) at ($ (k2) + (v1) + (v2) $) {};
\node (b321) at ($ (k2) + 2*(v1) + (v2) $) {};
\node (b131) at ($ (k2) + 2*(v2) $) {};
\node (b231) at ($ (k2) + (v1) + 2*(v2) $) {};
\node (b331) at ($ (k2) + 2*(v1) + 2*(v2) $) {};
\node (b112) at ($ (k2) + (v3) $) {};
\node (b212) at ($ (k2) + (v1) + (v3) $) {};
\node (b312) at ($ (k2) + 2*(v1) + (v3) $) {};
\node (b122) at ($ (k2) + (v2) + (v3) $) {};
\node (b222) at ($ (k2) + (v1) + (v2) + (v3) $) {};
\node (b322) at ($ (k2) + 2*(v1) + (v2) + (v3) $) {};
\node (b132) at ($ (k2) + 2*(v2) + (v3) $) {};
\node (b232) at ($ (k2) + (v1) + 2*(v2) + (v3) $) {};
\node (b332) at ($ (k2) + 2*(v1) + 2*(v2) + (v3) $) {};
\node (b113) at ($ (k2) + 2*(v3) $) {};
\node (b213) at ($ (k2) + (v1) + 2*(v3) $) {};
\node (b313) at ($ (k2) + 2*(v1) + 2*(v3) $) {};
\node (b123) at ($ (k2) + (v2) + 2*(v3) $) {};
\node (b223) at ($ (k2) + (v1) + (v2) + 2*(v3) $) {};
\node (b323) at ($ (k2) + 2*(v1) + (v2) + 2*(v3) $) {};
\node (b133) at ($ (k2) + 2*(v2) + 2*(v3) $) {};
\node (b233) at ($ (k2) + (v1) + 2*(v2) + 2*(v3) $) {};
\node (b333) at ($ (k2) + 2*(v1) + 2*(v2) + 2*(v3) $) {};
\node (c111) at (k3) {};
\node (c211) at ($ (k3) + (v1) $) {};
\node (c121) at ($ (k3) + (v2) $) {};
\node (c221) at ($ (k3) + (v1) + (v2) $) {};
\node (c131) at ($ (k3) + 2*(v2) $) {};
\node (c231) at ($ (k3) + (v1) + 2*(v2) $) {};
\node (c112) at ($ (k3) + (v3) $) {};
\node (c212) at ($ (k3) + (v1) + (v3) $) {};
\node (c312) at ($ (k3) + 2*(v1) + (v3) $) {};
\node (c122) at ($ (k3) + (v2) + (v3) $) {};
\node (c222) at ($ (k3) + (v1) + (v2) + (v3) $) {};
\node (c322) at ($ (k3) + 2*(v1) + (v2) + (v3) $) {};
\node (c132) at ($ (k3) + 2*(v2) + (v3) $) {};
\node (c232) at ($ (k3) + (v1) + 2*(v2) + (v3) $) {};
\node (c332) at ($ (k3) + 2*(v1) + 2*(v2) + (v3) $) {};
\node (c113) at ($ (k3) + 2*(v3) $) {};
\node (c213) at ($ (k3) + (v1) + 2*(v3) $) {};
\node (c313) at ($ (k3) + 2*(v1) + 2*(v3) $) {};
\node (c123) at ($ (k3) + (v2) + 2*(v3) $) {};
\node (c223) at ($ (k3) + (v1) + (v2) + 2*(v3) $) {};
\node (c323) at ($ (k3) + 2*(v1) + (v2) + 2*(v3) $) {};
\node (c133) at ($ (k3) + 2*(v2) + 2*(v3) $) {};
\node (c233) at ($ (k3) + (v1) + 2*(v2) + 2*(v3) $) {};
\node (c333) at ($ (k3) + 2*(v1) + 2*(v2) + 2*(v3) $) {};
\draw [edge] (c113) -- (c213) -- (c313);
\draw [edge] (c123) -- (c223) -- (c323);
\draw [edge] (c133) -- (c233) -- (c333);
\draw [edge] (c113) -- (c123) -- (c133);
\draw [edge] (c213) -- (c223) -- (c233);
\draw [edge] (c313) -- (c323) -- (c333);
\draw [edge] (c111) -- (c112) -- (c113);
\draw [edge] (c211) -- (c212) -- (c213);
\draw [edge] (c311) -- (c312) -- (c313);
\draw [edge] (c121) -- (c122) -- (c123);
\draw [edge] (c221) -- (c222) -- (c223);
\draw [edge] (c321) -- (c322) -- (c323);
\draw [edge] (c131) -- (c132) -- (c133);
\draw [edge] (c231) -- (c232) -- (c233);
\draw [edge] (c331) -- (c332) -- (c333);
\draw [edge] (c112) -- (c212) -- (c312);
\draw [edge] (c122) -- (c222) -- (c322);
\draw [edge] (c132) -- (c232) -- (c332);
\draw [edge] (c112) -- (c122) -- (c132);
\draw [edge] (c212) -- (c222) -- (c232);
\draw [edge] (c312) -- (c322) -- (c332);
\draw [edge] (c111) -- (c211) -- (c311);
\draw [edge] (c121) -- (c221) -- (c321);
\draw [edge] (c131) -- (c231) -- (c331);
\draw [edge] (c111) -- (c121) -- (c131);
\draw [edge] (c211) -- (c221) -- (c231);
\draw [edge] (b113) -- (b213) -- (b313);
\draw [edge] (b123) -- (b223) -- (b323);
\draw [edge] (b133) -- (b233) -- (b333);
\draw [edge] (b113) -- (b123) -- (b133);
\draw [edge] (b213) -- (b223) -- (b233);
\draw [edge] (b313) -- (b323) -- (b333);
\draw [edge] (b111) -- (b112) -- (b113);
\draw [edge] (b211) -- (b212) -- (b213);
\draw [edge] (b311) -- (b312) -- (b313);
\draw [edge] (b121) -- (b122) -- (b123);
\draw [edge] (b221) -- (b222) -- (b223);
\draw [edge] (b321) -- (b322) -- (b323);
\draw [edge] (b131) -- (b132) -- (b133);
\draw [edge] (b231) -- (b232) -- (b233);
\draw [edge] (b331) -- (b332) -- (b333);
\draw [edge] (b112) -- (b212) -- (b312);
\draw [edge] (b122) -- (b222) -- (b322);
\draw [edge] (b132) -- (b232) -- (b332);
\draw [edge] (b112) -- (b122) -- (b132);
\draw [edge] (b212) -- (b222) -- (b232);
\draw [edge] (b312) -- (b322) -- (b332);
\draw [edge] (a113) -- (a213) -- (a313);
\draw [edge] (a123) -- (a223) -- (a323);
\draw [edge] (a133) -- (a233) -- (a333);
\draw [edge] (a113) -- (a123) -- (a133);
\draw [edge] (a213) -- (a223) -- (a233);
\draw [edge] (a313) -- (a323) -- (a333);
\draw [edge] (a111) -- (a112) -- (a113);
\draw [edge] (a211) -- (a212) -- (a213);
\draw [edge] (a311) -- (a312) -- (a313);
\draw [edge] (a121) -- (a122) -- (a123);
\draw [edge] (a221) -- (a222) -- (a223);
\draw [edge] (a321) -- (a322) -- (a323);
\draw [edge] (a131) -- (a132) -- (a133);
\draw [edge] (a231) -- (a232) -- (a233);
\draw [edge] (a331) -- (a332) -- (a333);
\draw [edge] (a112) -- (a212) -- (a312);
\draw [edge] (a122) -- (a222) -- (a322);
\draw [edge] (a132) -- (a232) -- (a332);
\draw [edge] (a112) -- (a122) -- (a132);
\draw [edge] (a212) -- (a222) -- (a232);
\draw [edge] (a312) -- (a322) -- (a332);
\draw [edge] (a111) -- (a211) -- (a311);
\draw [edge] (a121) -- (a221) -- (a321);
\draw [edge] (a131) -- (a231) -- (a331);
\draw [edge] (a111) -- (a121) -- (a131);
\draw [edge] (a211) -- (a221) -- (a231);
\draw [edge] (b111) -- (b121) -- (b131);
\draw [edge] (b211) -- (b221) -- (b231);
\draw [edge] (b111) -- (b211) -- (b311);
\draw [edge] (b121) -- (b221) -- (b321);
\draw [edge] (b131) -- (b231) -- (b331);
\end{tikzpicture}
}
\caption{The Gray $(27_3)$ configuration $\mathcal{C}_0$ and the corresponding $\mathcal{C}_1$.}
\label{fig:gray}
\end{figure}
Let $\mathcal{C}_0$ be the Gray configuration. By one application of Lemma~14 we obtain a configuration $\mathcal{C}_1$ (see Figure~\ref{fig:gray}(b)) which is point- and line-splittable.
\end{example}

\section{Conclusion}
Theorems 9 and 10, Corollary 8, and our experimental investigations (see periodic behaviour of the last column of Table 1 past $n = 9$) of splittability of cyclic Haar graphs led us to the following conjecture.
\begin{conjecture}
A cyclic $(n_3)$ configuration is unsplittable if and only if its Levi graph belongs to one of the following three infinite families:
\begin{enumerate}
\item $H(n, \{0, 1, 3\})$ for $n \geq 7$;
\item $H(3n, \{0, 1, n\})$ for $n \geq 2$;
\item $H(3n, \{0, 1, n + 1\})$ for $n \geq 4$ where $n \not\equiv 0 \pmod 3$.
\end{enumerate}
\end{conjecture}

To show that all other cyclic $(n_3)$ configurations are splittable, we expect that the method used in the proof of Theorem 6, Corollary 7 and 8 can be extended. Nedela and Škoviera~\cite{NS} showed a nice property of cubic graphs with respect to the cyclic connectivity. Their result is likely to have applications in splittablity.

In Section~\ref{sec:sgc} we have shown how to construct geometric point- and line-splittable $(n_k)$ configuration for any $k$.
However, we were not able to obtain any splittable cyclic $(n_k)$ configuration for $k \geq 4$ so far. Therefore, we pose the following claim.
\begin{conjecture}
All cyclic $(n_k)$ configurations for $k \geq 4$ are unsplittable.
\end{conjecture}

Notions of splittable and unsplittable configurations have been defined via associated graphs. Since splittability is a property
of combinatorial configurations, it can be extended from bipartite graphs of girth at least 6 to more general graphs.
We expect that results concerning cyclic connectivity such as those presented in \cite{NS} will play an important
role in such investigations.

Note that cyclic Haar graphs have girth at most 6 and form a special class of bicirculants \cite{tomo2007}.
However, there exist other bicirculants with girth greater than 6. The corresponding configurations have been
investigated in \cite{BPZ2005,BGPZ2006}. One way of extending this study is on the one hand to consider splittability of
these more general bicirculants and on the other hand to study tricirculants~\cite{kkmw2012}, tetracirculants and beyond \cite{Fre2013}.
In the language of configurations, they can be described as special classes of polycyclic configurations \cite{boben2003}.

\section{Acknowledgement}
The work was 
supported in part by a grant from the Picker Institute at Colgate University and the ARRS of Slovenia, grants P1-0294, N1-0011, N1-0032, J1-6720, and J1-7051. We would like to thank Istv{\'a}n Est{\'e}lyi and Jaka Kranjc for fruitful discussions during the preparation of this paper and the two anonymous referees for useful remarks and suggestions that have improved the quality of the paper.

\bibliography{splittable}
\bibliographystyle{amcjoucc}

\end{document}